\documentclass[11pt,a4paper,headings=standardclasses]{scrarticle}
\usepackage[automark]{scrlayer-scrpage}

\usepackage{amsmath,amsthm,amssymb,commath}
\usepackage[font=small,labelfont=bf]{caption}
\usepackage[labelformat=simple]{subcaption}
\usepackage{enumitem}
\usepackage{graphicx}
\usepackage{float}
\usepackage{xcolor}
\usepackage[margin=1 in ]{geometry}
\usepackage{setspace}
\usepackage{tikz}
\usepackage{tikz-cd}
\usetikzlibrary{arrows}
\usetikzlibrary{arrows.meta}
\usetikzlibrary{positioning,calc}
\usetikzlibrary{cd, babel}

\usepackage[colorlinks]{hyperref}
\usepackage[capitalize]{cleveref}

\hypersetup{
	colorlinks=true,
	linkcolor=blue,
	filecolor=blue,
	citecolor = blue,    
	urlcolor=blue,
}

\newtheorem{thm}{Theorem}[section]
\newtheorem{cor}[thm]{Corollary}
\newtheorem{prop}[thm]{Proposition}
\newtheorem{lem}[thm]{Lemma}

\theoremstyle{definition}
\newtheorem{defn}[thm]{Definition}
\newtheorem{exmp}[thm]{Example}

\theoremstyle{remark}
\newtheorem{rem}[thm]{Remark}

\newcommand{\R}{\mathbb{R}}
\newcommand{\Z}{\mathbb{Z}}

\newcommand{\id}{\operatorname{id}}
\newcommand{\im}{\operatorname{im}}
\newcommand{\F}{\mathbb{F}}
\newcommand{\nerve}{\operatorname{\mathcal{N}}}

\newcommand{\VR}{\operatorname{VR}}
\newcommand{\Cech}{\operatorname{\check{C}}}
\newcommand{\Cub}{\operatorname{Cub}}

\newcommand{\Cyc}{\operatorname{Cyc}}
\newcommand{\Torus}{\operatorname{\mathbb{T}}}
\renewcommand{\epsilon}{\varepsilon}
\newcommand{\DC}{\operatorname{DC}}
\newcommand{\RDC}{\operatorname{RDC}}

\title{Distance Matrices of Ordered Point Clouds and Their Persistent Homology}
\author{David Hien}

\begin{document}
\maketitle

\begin{abstract}
	The distance matrix of a finite point cloud can be visualized as a heatmap. When the data arise from a time series, the sublevel sets of this image are known as recurrence plots and are widely used in time series analysis. Motivated by this perspective, we establish a relationship between the distance-matrix filtration of the time series and the \v{C}ech (or Vietoris--Rips) filtration of its state-space embedding in the form of a degree-one chain map.
	We study the induced maps in homology, showing that the map from $H_0$ into $H_1$ is essentially surjective and providing an example where the map from $H_1$ into $H_2$ is nontrivial.
	These chain maps can be applied to simplify image persistence computations arising in the computation of cycling signatures, a topological tool for time series analysis.
	Moreover, these computations yield finer information that allows the analysis of transitions between different types of cycling motion.
\end{abstract}

\section{Introduction}
Analyzing time series data in order to understand an underlying governing process is a fundamental task across scientific fields.
A common phenomenon in such data is recurrence: an observed trajectory returns close to a previously visited state.
Recurrence plots \cite{Eckmann1987recurrenceplots} provide a visual representation of this information. 
Given a time series $x_1,\dots, x_n $ and a metric $ d(x,y) $, the recurrence plot at scale $\epsilon$ is an $n\times n$-pixel black-and-white image; the pixel at position~$(i,j)$ is black if the distance $d(x_i,x_j)$ is less than $\epsilon$ and white otherwise.
Thus, a black pixel away from the diagonal indicates recurrence because the time series visits nearby locations at distinct times.
Recurrence plots are known to reveal qualitative features of the dynamics; for example, different patterns can be related to periodicity, quasi-periodicity, and chaos.
Furthermore, recurrence plots are studied from a quantitative point of view \cite{Marwan2014rqa}.

From the viewpoint of topological data analysis, $\epsilon$ can be regarded as a filtration parameter.
More precisely, the distance matrix of the time series defines a filtration, where the recurrence plot at scale $\epsilon$ is the image at that parameter. 
Recent work by Ichinomiya \cite{Ichinomiya2023,Ichinomiya2025} describes an experimental pipeline for time series analysis based on the persistent homology of the recurrence plot filtration.

At the same time, the embedded points $x_1,\dots,x_n$ in state space also give rise to filtrations such as the offset, \v{C}ech, and Vietoris--Rips filtrations.
Even though they are derived from the same time series, a recurrence plot filtration and the point cloud traced out by the time series are a priori topologically unrelated.

The main contribution of this paper is a theoretical link between the distance matrix and spatial filtrations in the form of a degree-one chain map.
Consider an ordered point cloud $\Gamma=(x_1,\dots,x_n)$ in a metric space $(X,d)$. Our main examples are time series obtained by sampling curves, such as solutions of differential equations.
The basic observation of this paper is that close returns in the recurrence plot naturally define 1-cycles in the chain complex of the geometric complex of the embedded point cloud. 
This observation can be turned into a chain map as follows. 
The Rips distance complex (\cref{def:rips-distance-complex}) of $\Gamma$ is the $\mathcal{V}$-construction of the upper triangular part of the distance matrix. 
This means every matrix entry $(i,j)$ with $i\leq j$ becomes a vertex, with filtration value equal to the respective matrix entry.
Horizontal and vertical edges between neighboring entries are added whenever possible, as are squares $[i,i+1]\times [j,j+1]$.
For filtration values $ r > S_\Gamma = \max_{1\leq i<n} d(x_i,x_{i+1}) $, we construct a chain map as follows:
Given a vertex $(i,j)$ in $\RDC_r(\Gamma)$ (which corresponds to a close return in the recurrence plot), we assign the cycle in the Vietoris--Rips complex that follows the order of the point cloud via $ x_i,x_{i+1},\dots, x_j $ and uses the edge between $x_i$ and $x_j$ to close the cycle.
This extends a degree-one chain map
\[
	\tilde \psi \colon C_\bullet(\RDC(\Gamma)) \longrightarrow C_{\bullet+1}(\VR(\Gamma))
\]
which we call the distance-to-Rips map.
As we show (see \cref{prop:rdc-is-dc}), the construction of the Rips distance complex and the distance-to-Rips map can be regarded as special cases of the more general constructions of the \v{C}ech distance complex (\cref{def:distance-complex}) and the distance-to-\v{C}ech map (\cref{def:distance-to-cech-map}):
\[
	\psi\colon \DC(\Gamma) \rightarrow \Cech(\Gamma).
\]
A theoretical link between the constructions is provided by the Kuratowski embedding, which embeds a point cloud (or, more generally, a compact metric space) into an ambient space where \v{C}ech and Vietoris--Rips complexes agree up to a rescaling \cite{LimMemoliOkutan2024}.

From a practical point of view, stability of these constructions is a desirable property.
We show that the distance complexes and distance-to-\v{C}ech maps of $\epsilon$-perturbed point clouds are $\epsilon$-interleaved.
It follows that the induced persistence morphisms are $\epsilon$-interleaved, as are their kernel, image, and cokernel persistence modules. 
With field coefficients, the associated barcodes are then stable with respect to the bottleneck distance.

We then study the induced maps in homology. The first main result is that the map
\[
	H_0(\psi) \colon H_0(\DC(\Gamma)) \longrightarrow H_{1}(\Cech(\Gamma))
\]
is surjective. 
This has an immediate computational application:
If one wants to compute the image persistence \cite{CohenSteiner2009kernels} of a morphism $ \phi$ with a Vietoris--Rips or a \v{C}ech complex as its domain, one can precompose with $ \psi $ and compute $ \im H_0(\phi\circ \psi) $ instead.
The map $H_0(\psi)$ need not be injective, even after removing the component corresponding to the diagonal.
This is interesting in the context of time series analysis:
While every homological feature is generated by close returns, distinct close returns can wrap around the same feature and thus exhibit the same `homological' behavior.

While the map $ H_0(\psi) $ is fairly transparent, the behavior of $ H_1(\psi) $ is more subtle;
we explore it through examples. 
A time series that winds around an $S^2$ provides an example of a nontrivial $ H_1(\psi) $.
Unlike in degree zero, the map $ H_1(\psi) $ is generally not surjective.

The distance complex and the distance-to-\v{C}ech map can be applied to time series analysis in the context of cycling signatures \cite{Bauer2025cyclingSignatures}.
Given a time series (i.e., an ordered point cloud) $ \Gamma $ and a neighborhood $Y$ of $\Gamma$, the cycling signature of a segment $\gamma$ of $\Gamma$ is defined as the image of the inclusion-induced map
\[
	H_1(O(\gamma) \rightarrow Y).
\]
The idea of the cycling signature is to compare temporal information in the form of segments with spatial topological information captured in the space $Y$.

We can make use of the distance-to-\v{C}ech map in three ways.
First, after suitable discretization, the map in the definition of the cycling signature has a \v{C}ech filtration as its domain.
We can thus compute the cycling signature by precomposing this map with the distance-to-\v{C}ech map.
This yields a substantial advantage in some examples.
Second, and more importantly, the cycles identified by the distance-to-\v{C}ech map provide additional information that is not captured by the signatures alone.
We demonstrate this by analyzing transitions between different cycling motions in a three-wing chaotic attractor.
Third, the chain map yields insight into the relationship among cycling signatures, recurrence plots, and the pipeline of \cite{Ichinomiya2023,Ichinomiya2025}.

In future work, we plan to extend the notions of the distance complex and the distance-to-\v{C}ech map to continuous curves.
We note that some relevant definitions have already appeared in \cite{diss}.
One reason for this is to obtain better stability results. 
Another possible extension, essentially also suggested in \cite{binnie2026chordaldistancetransform}, is to generalize these constructions to distance triples, for example by using the volume of the enclosed simplex or the radius of the minimal enclosing ball.
Furthermore, it would be interesting to have a better understanding of how nontrivial degree-one maps arise.
On the practical side, it would be interesting to explore the combination of the distance-to-\v{C}ech map with Ichinomiya's pipeline without the cycling signature as an intermediate step.

\subsubsection*{Related Work}
Recurrence plots were introduced in \cite{Eckmann1987recurrenceplots}; 
see \cite{Marwan2007recurrenceplots,Marwan2014rqa} for surveys of recurrence plots and recurrence quantification analysis.
There are several constructions of networks from recurrence plots in the time series literature \cite{ZhangSmall2006,Donner2010recurrencenetworks}.

Persistent homology has been applied to point clouds constructed from time series in a variety of settings \cite{perea2015sliding,gidea2018topological,topaz2015topological,kim2020spatio,kramar2015pattern}.
The setting most closely related to this work is when a point cloud is generated from a sampled trajectory of a dynamical system, for example, in
\cite{Maletic2016persistent,strommen2022climate}.

Sublevel-set persistence of grayscale images is commonly modeled by cubical filtrations; in particular, the dual $\mathcal V$- and $\mathcal T$-constructions and their relationship are studied in \cite{Robins2011,Bleile2022}.
Ichinomiya applies cubical persistent homology to time-ordered distance matrices \cite{Ichinomiya2023} and subsequently uses persistence images \cite{Adams2017persistenceimages} to vectorize the resulting persistence diagrams for classifying time series \cite{Ichinomiya2025}.
Note that the persistence of images differs from image persistence, by which we mean the persistence module obtained as the image of a morphism of persistence modules \cite{CohenSteiner2009kernels}.

In addition to our work and \cite{Ichinomiya2023,Ichinomiya2025}, the persistent homology of distance matrices is studied in \cite{binnie2026chordaldistancetransform} in the context of shape analysis.
There, a persistence module is associated to a loop $S^1 \rightarrow \R^d$ as follows:
First, one maps all pairs $ S^1\times S^1 $ to the distance of the corresponding points on the loop.
Since this map does not depend on the order of the pair, it descends to a map on $F_2(S^1)$, the unordered configuration space of two points in $S^1$.
The resulting map $F_2(S^1) \rightarrow \R$ is called the \emph{chordal distance transform}.
They then pursue a geometric program, relating critical points of the chordal distance transform to the geometry of the loop.
Notably, neither line of work considers linking the filtration of the distance matrix to the corresponding spatial filtration via a chain map.

\subsubsection*{Outline}
In \cref{sec:preliminaries}, we review some standard topological constructions and introduce some notation.
In \cref{sec:distance-complex}, we give the general definition of the distance complex, which is compatible with the \v{C}ech filtration.
In \cref{sec:distance-to-cech-map}, we introduce the distance-to-\v{C}ech map.
We prove stability with respect to pointwise perturbations and explore the induced maps in homology. We furthermore introduce the Rips distance complex and the distance-to-Rips map and relate them to the \v{C}ech setting using the Kuratowski embedding.
In \cref{sec:application}, we apply the distance-to-\v{C}ech map to time series analysis via the cycling signature.

\section{Preliminaries}
\label{sec:preliminaries}

Let $I\subset \R$ be an interval. 
By an $I$-filtered simplicial complex, chain complex, or vector space,
we mean a family $(F_t)_{t\in I}$ of objects of the corresponding type, together with inclusions $F_{s,t}\colon F_s\to F_t$ for every $s<t$ in $I$.
A morphism between such $I$-filtered objects $F$ and $G$ is a collection of maps $F_t\rightarrow G_t$ for every $t\in I$ that commutes with the respective inclusions.

\subsection{Simplicial Complexes}

Throughout this section, let $(X,d)$ be a metric space.
We write $B(x,r)=\{y\in X\mid d(x,y)<r\}$ for the open ball of radius $r$ about $x$.

\begin{defn}
	An \emph{indexed family} in $Z$ is a map $A\colon I\rightarrow Z$, written $A=(z_i)_{i\in I}$.
	Its image is the subset
	\[
		\im A=\{z_i\mid i\in I\}\subset Z.
	\]
\end{defn}

\begin{defn}
	An \emph{indexed cover} of a topological space $Y$ is a family $\mathcal U=(U_i)_{i\in I}$ of subsets of $Y$ such that $Y=\bigcup_{i\in I}U_i$.
	The \emph{nerve} of $\mathcal U$ is the simplicial complex
	\[
		\nerve(\mathcal U) = \left\{ \emptyset\neq\sigma\subset I \;\middle|\; \sigma\text{ is finite and } \bigcap_{i\in\sigma}U_i\neq\emptyset \right\}.
	\]
	An ordinary cover $\mathcal U$ is regarded as the family $(U)_{U\in\mathcal U}$ indexed by itself.
\end{defn}

\begin{defn}
	The \emph{\v{C}ech complex} of an indexed family $A=(x_i)_{i\in I}$ in $X$ at radius $r\geq 0$ is
	\[
		\Cech_r(A) = \nerve\bigl((B(x_i,r))_{i\in I}\bigr).
	\]
	The \v{C}ech filtration is the $\R$-filtered simplicial complex $\Cech(A)=(\Cech_r(A))_{r\geq 0}$.
	If $A\subset X$ is an ordinary subset, we regard it as the family $(a)_{a\in A}$ indexed by itself.
\end{defn}
The indexed \v{C}ech complex may have several vertices corresponding to the same point of $X$.
For example, if $X=\{\ast\}$ and $r>0$, then $\Cech_r((\ast,\ast))$ is an edge, whereas $\Cech_r(\{\ast\})$ is a single vertex.
However, the indexed and unindexed versions of the nerve are homotopy equivalent.
We give a brief argument using contiguity (taken from \cite[Section 5]{Virk2021RipsDowker}).
Recall that simplicial maps $f,g\colon K\rightarrow L$ are contiguous if for every simplex $\sigma$ in $K$, $f(\sigma)\cup g(\sigma)$ is a simplex in $L$.
By \cite[Theorem 12.5]{munkres1984elements}, chain maps induced by contiguous simplicial maps are chain homotopic.

\begin{lem}
	Let $\mathcal U=(U_i)_{i\in I}$ be an indexed cover and
	$\widetilde{\mathcal U}=\im\mathcal U=\{U_i\mid i\in I\}$.
	Then $\nerve(\widetilde{\mathcal U})\simeq\nerve(\mathcal U)$.
\end{lem}
\begin{proof}
	The map
	$p\colon I\rightarrow\widetilde{\mathcal U}$, $i\mapsto U_i$, is
	surjective.
	We can thus find a section
	$s\colon\widetilde{\mathcal U}\rightarrow I$ such that
	$p\circ s=\id_{\widetilde{\mathcal U}}$.
	Both maps induce simplicial maps on the corresponding nerves.
	Furthermore, $s\circ p$ and $\id_{\nerve(\mathcal U)}$ are contiguous:
	for every $\sigma\in\nerve(\mathcal U)$,
	\[
		\bigcap_{i\in\sigma}U_i
		=
		\bigcap_{i\in\sigma\cup(s\circ p)(\sigma)}U_i.
	\]
	Therefore, $p$ and $s$ give the required homotopy equivalence.
\end{proof}

\begin{defn}
	The \emph{Vietoris--Rips complex} of an indexed family $A=(x_i)_{i\in I}$ at radius $r$ is the simplicial complex
	\[
		\VR_r(A) = \{ \emptyset\neq \sigma\subset I\mid \sigma \text{ finite } \text{ and } \max_{i,j\in \sigma} d(x_i,x_j) < r \}.
	\]
	The Vietoris--Rips filtration is the $\R$-filtered simplicial complex $\VR(A)=(\VR_r(A))_{r\geq 0}$.
	As above, an ordinary subset $A\subset X$ is regarded as the family $(a)_{a\in A}$ indexed by itself.
\end{defn}

\subsection{Cubical Complexes}
\label{subsec:cubical-complexes}

We briefly recall cubical (grid) complexes in dimension two and refer to \cite{kaczynski2004computational} for details.

An interval of the form $I = [k,k+1]$, $k\in \Z$, is called a non-degenerate elementary interval.
An interval of the form $I = [k,k]$, $k\in \Z$, is called a degenerate elementary interval, and we sometimes simply write $[k]$. An elementary cube in $\R^2$ is a set of the form
\[
	Q = I_1\times I_2
\]
where $I_1$ and $I_2$ are elementary intervals. 
Its dimension is the number of non-degenerate intervals $I_k$.
An elementary cube $P$ is a face of an elementary cube $Q$ if $P\subset Q$.

The set of all elementary cubes in $\R^2$ is
\[
	\Cub(\R^2) = \bigcup_{k=0}^2 \Cub_k(\R^2),
\]
where $\Cub_{k}(\R^2)$ denotes the set of all such $ k $-dimensional elementary cubes.

A set $K\subset \Cub(\R^2)$ is called a \emph{cubical (grid) complex} if $K $ contains the faces of all its cubes. Its geometric realization is $|K| = \bigcup_{Q\in K} Q$.

Fix a coefficient ring $R$ and let $K$ be a cubical complex.
We denote by $C_k(K)$ the free $R$-module generated by the $k$-dimensional elementary cubes in $K$.
The boundary operator $\partial\colon C_k(K) \rightarrow C_{k-1}(K)$ is the linear extension of the following operations:
\begin{itemize}
	\item $ \partial_0 \colon C_0(K) \rightarrow C_{-1}(K)$, where $\partial_0 \left([i] \times [j]\right) = 0$.
	
	\item $ \partial_1 \colon C_1(K) \rightarrow C_{0}(K)$, where
	\begin{itemize}
		\item $\partial_1 \left( [i] \times [j,j+1]\right) = [i]\times [j+1] - [i]\times [j]$, and
		\item $\partial_1 \left( [i,i+1] \times [j]\right) = [i+1]\times [j] - [i]\times [j]$.
	\end{itemize}
	
	\item $ \partial_2 \colon C_2(K) \rightarrow C_{1}(K)$, where
	\begin{align*}
		\partial_2([i,i+1] \times [j,j+1])
		&= [i+1]\times [j,j+1] - [i]\times [j,j+1] \\
		&\quad \quad - [i,i+1]\times [j+1] + [i,i+1]\times [j].
	\end{align*}
\end{itemize}
Together with the $ C_k(K) $, this defines the cubical chain complex of $K$. 
It can be shown that the cubical homology of a cubical complex is canonically isomorphic to the cellular homology of $|K|$.

\section{The Distance Complex}\label{sec:distance-complex}

The distance complex filtration captures the sublevel set persistent homology of the distance matrix of an ordered point cloud in a way that is consistent with the \v{C}ech complex.
Related constructions have appeared in \cite{Ichinomiya2025} and in \cite{binnie2026chordaldistancetransform};
see \cref{rem:comparison-with-ichinomiya} and the introduction for details.

Throughout this section, $ (X,d) $ denotes a metric space.
We abstract time series via the notion of an ordered point cloud.

\begin{defn}
	An \emph{ordered point cloud} in $(X,d)$ is an indexed family $\Gamma=(x_i)_{i\in I}$ where $(I,<)$ is totally ordered.
\end{defn}

Suppose $\Gamma=(x_1,\dots,x_n)$ is an ordered point cloud in $(X,d)$.
We associate a filtered cubical grid complex to the distance matrix $D_{i,j}=d(x_i,x_j)$ as follows.
We take a vertex for every entry in the matrix and assign it a filtration value equal to half the matrix entry.
We add the horizontal edge between $ (i,j) $ and $(i+1,j)$ once the simplex $\{i,i+1,j\}$ is contained in $\Cech_r(\Gamma)$, or, equivalently,  once the triple intersection
\[
	B(x_i,r) \cap B(x_{i+1}, r) \cap B(x_j, r)
\]
is nonempty. We do the same for vertical edges between $ (j,i) $ and $(j, i+1)$.
Lastly, we add the square with corners $ (i,j) $ and $(i+1,j+1)$ if the corresponding index set $\{i,i+1,j,j+1\}$ spans a simplex in $\Cech_r(\Gamma)$.
This is summarized in the following definition.

\begin{defn}\label{def:distance-complex}
	Let $\Gamma=(x_1,\dots,x_n)$ be an ordered point cloud in a metric space $(X,d)$.
	The \emph{(upper right) distance complex} of $\Gamma$ is the filtered cubical complex
	\[
	\DC(\Gamma)=\bigl(\DC_r(\Gamma)\bigr)_{r\geq 0},
	\]
	where
	\begin{equation}\label{eq:distance-complex}
		\DC_r(\Gamma) = \left\{ [i,j]\times [k,l] \in \Cub([1,n]^2) \mid j\leq k \text{ and }
		\{ i,j,k,l \} \in \Cech_r(\Gamma) \right\}.
	\end{equation}
\end{defn}

The condition $j\le k$ ensures that every vertex of the elementary cube is contained in the upper-triangular part of the grid.

\begin{exmp}
	Let $\Gamma=(x_1,\dots,x_6)$ be the time series obtained by traversing the vertices of a regular hexagon inscribed in the unit circle counterclockwise, as shown in \cref{fig:dist-complex-example}. Consider $r=1$.
	The \v{C}ech complex has all edges except between the three opposite pairs, and a triangle for each triple of consecutive points.
	Therefore, the distance complex is the upper triangular part of the matrix, with three gray vertices corresponding to the opposite pairs.
	Furthermore, it is easy to check that all possible edges and squares between the black vertices are present.
	For example, the four squares of the form $ [i,i+1]\times [i+1,i+2]$ correspond to the intersections
	\[
		B(x_i,r) \cap B(x_{i+1}, r) \cap B(x_{i+1}, r) \cap B(x_{i+2}, r), \qquad i = 1,2,3,4.
	\]
	Note that two balls in this quadruple intersection are the same.
\end{exmp}

From a practical point of view, stability is an important property.
Stability of the distance complex does not come as a surprise, given that the \v{C}ech filtration is stable \cite{Chazal2009}.
In this work, we only consider pointwise perturbations of the ordered point cloud with respect to the supremum distance
\[
	d_\infty(\Gamma,\Gamma') = \max_{1\leq i\leq n} d(x_i,x_i'),
\]
where $ \Gamma = (x_1,\dots, x_n) $ and $ \Gamma' = (x_1',\dots, x_n') $ are time series in $(X,d)$.
This significantly simplifies the analysis since we need only consider a single correspondence between the point clouds instead of the collection of all possible correspondences.

\begin{rem}
	In the context of continuous curves, stability results with respect to the Fréchet distance were obtained in \cite{binnie2026chordaldistancetransform}.
	While these results can be adapted to the distance complex without too much work, applying them to the distance-to-\v{C}ech map is not straightforward.
	We therefore restrict our attention to pointwise perturbations.	
\end{rem}

\begin{lem}[Stability]\label{lem:distance-complex-stability}
	Let $\Gamma_1 $ and $\Gamma_2 $ be ordered point clouds of length $n$ in a metric space $(X,d)$. 
	Let $\varepsilon = d_\infty(\Gamma_1,\Gamma_2)$.
	Then
	\[
		\DC_r(\Gamma_1)\subset \DC_{r+\varepsilon}(\Gamma_2) \qquad \text{ and }\qquad \DC_r(\Gamma_2)\subset \DC_{r+\varepsilon}(\Gamma_1).
	\]
\end{lem}
\begin{proof}
	If $\sigma\in\Cech_r(\Gamma_1)$, choose $y\in\bigcap_{i\in\sigma}B(\Gamma_1(i),r)$.
	For every $i\in\sigma$, the triangle inequality gives
	$d(y,\Gamma_2(i))<r+\varepsilon$, so $\sigma\in\Cech_{r+\varepsilon}(\Gamma_2)$.
	The claim now follows from \cref{eq:distance-complex}.
\end{proof}

In this work, we are mainly concerned with ordered point clouds that arise from sampling a continuous curve.

\begin{defn}
	The \emph{sample radius} of an ordered point cloud $\Gamma = (x_1,\dots,x_n)$ is
	\[
		R_\Gamma = \inf \{r\geq 0 \mid B(x_{i-1},r) \cap B(x_i,r) \neq \emptyset \text{ for all } 2\leq i\leq n\}.
	\]
\end{defn}

In a Euclidean space, the sample radius is half the maximal distance between consecutive points.

\begin{lem}
	Let $\Gamma=(x_1,\dots,x_n)$ be an ordered point cloud with sample radius $R_\Gamma$.
	\begin{enumerate}[label = (\roman*)]
		\item For all $r>0$ and $1\leq i\leq n$, we have $ [i,i]\times [i,i] \in \DC_r(\Gamma)$.
		\item For $ r > R_\Gamma $ and $1\leq i<n$, we have
		\[
			[i,i]\times [i,i+1] \in \DC_r(\Gamma) \qquad \text{and}\qquad [i,i+1]\times [i+1,i+1] \in \DC_r(\Gamma).
		\]
		\item For $ r > R_\Gamma $, all diagonal vertices $ [i]\times [i]$, $i = 1,\dots, n$, are in the same connected component of $ \DC_r(\Gamma) $.
	\end{enumerate}
\end{lem}
\begin{proof}
	\emph{(i)} and \emph{(ii)} are immediate. For \emph{(iii)}, note that the edges in \emph{(ii)} connect the diagonal entries.
\end{proof}

The diagonal entries of a distance matrix are zero; therefore, the diagonal vertices are in the filtration for $r>0$.
If the time series is sampled finely enough, the vertices near the diagonal also enter at a small scale, connecting all the diagonal vertices.
Off-diagonal components correspond to close returns, i.e., pairs of distant time indices whose states are close in $X$.

\section{The Distance-to-\v{C}ech Map}
\label{sec:distance-to-cech-map}

The distance complex and the \v{C}ech complex encode different aspects of the same data. 
The complex $\DC_r(\Gamma)$ is organized by pairs of time indices and records where the trajectory returns close to itself. 
The complex $\Cech_r(\Gamma)$ has vertices indexed by the time indices and records the geometry of the corresponding balls in state space. 
In this section, we relate these complexes.

Throughout this section, we fix an ordered point cloud $\Gamma=(x_1,\dots,x_n)$ in a metric space $(X,d)$.
We furthermore consider simplicial homology (and the associated chain complex) with coefficients in an arbitrary ring.
By a degree-one chain map between chain complexes $C$ and $D$, we mean a chain map that maps $ C_k\rightarrow D_{k+1} $ and commutes with the boundary operator.

\begin{defn}
	For $1\leq i\leq j\leq n$, define the \emph{curve chain}
	\[
	c_\Gamma(i,j) = \sum_{k=i+1}^j [k-1,k].
	\]
\end{defn}
The curve chain `traces' the curve from $x_i$ to $x_j$ and has the endpoints as its boundary.
For $r>R_\Gamma$, this chain is contained in $C(\Cech_r(\Gamma))$ and satisfies $ \partial c_\Gamma(i,j) = [j] - [i]$.

\begin{figure}
	\centering
	\begin{subfigure}[b]{0.32\textwidth}
		\centering
		\begin{tikzpicture}[scale=.6]
			\begin{scope}
				\path (-3,-3) rectangle (3,3);
				
				\draw[->,thick] (-2.8,0) -- (2.8,0);
				\draw[->,thick] (0,-2.8) -- (0,2.8);
				
				\draw[lightgray,thin] (0,0) circle (1.7);
				
				\foreach \k in {0,...,5} {
					\coordinate (P\k) at ({1.7*cos(60*\k)},{1.7*sin(60*\k)});
					\fill (P\k) circle (4pt);
				}
				
				\draw[gray,dashed,thick]
				(P0) -- node[midway,below] {\tiny $\sqrt{3}$} (P2);
				
				\node[above right]       at (P0) {\footnotesize $x_1$};
				\node[above right] at (P1) {\footnotesize $x_2$};
				\node[above left]  at (P2) {\footnotesize $x_3$};
				\node[above left]        at (P3) {\footnotesize $x_4$};
				\node[below left]  at (P4) {\footnotesize $x_5$};
				\node[below right] at (P5) {\footnotesize $x_6$};
				
			\end{scope}
		\end{tikzpicture}
		\caption{$\Gamma$}
	\end{subfigure}\hfill
	\begin{subfigure}[b]{0.32\textwidth}
		\centering
		\begin{tikzpicture}[scale=.6,
			defaultnode/.style={circle, fill=black, inner sep=2pt},
			specialnode/.style={circle, fill=lightgray, inner sep=2pt},
			thickline/.style={thick}
			]
			
			\fill[gray!50] (2, -1) -- (3, -1) -- (3, -2) -- (2, -2) -- cycle;
			\fill[gray!50] (3, -2) -- (4, -2) -- (4, -3) -- (3, -3) -- cycle;
			\fill[gray!50] (4, -3) -- (5, -3) -- (5, -4) -- (4, -4) -- cycle;
			\fill[gray!50] (5, -4) -- (6, -4) -- (6, -5) -- (5, -5) -- cycle;
			
			\fill[gray!20] (3, -1) -- (4, -1) -- (4, -2) -- (3, -2) -- cycle;
			\fill[gray!20] (4, -1) -- (5, -1) -- (5, -2) -- (4, -2) -- cycle;
			\fill[gray!20] (5, -1) -- (6, -1) -- (6, -2) -- (5, -2) -- cycle;
			
			\fill[gray!20] (4, -2) -- (5, -2) -- (5, -3) -- (4, -3) -- cycle;
			\fill[gray!20] (5, -2) -- (6, -2) -- (6, -3) -- (5, -3) -- cycle;
			
			\fill[gray!20] (5, -3) -- (6, -3) -- (6, -4) -- (5, -4) -- cycle;
			
			\newcommand{\isspecial}[2]{
				\def\result{0}
				\ifnum#1=1\relax\ifnum#2=4\relax\def\result{1}\fi\fi
				\ifnum#1=2\relax\ifnum#2=5\relax\def\result{1}\fi\fi
				\ifnum#1=3\relax\ifnum#2=6\relax\def\result{1}\fi\fi
			}
			
			\foreach \i in {1,...,6} {
				\foreach \j in {1,...,6} {
					\ifnum\i>\j
					\else
					\isspecial{\i}{\j}
					\ifnum\result=1
					\node[specialnode] (n\i\j) at (\j, -\i) {};
					\else
					\node[defaultnode] (n\i\j) at (\j, -\i) {};
					\fi
					\fi
				}
			}

			\draw[thickline] (n11) -- (n12);
			\draw[thickline] (n12) -- (n13);
			\draw[thickline,lightgray] (n13) -- (n14);
			\draw[thickline,lightgray] (n14) -- (n15);
			\draw[thickline] (n15) -- (n16);
			
			\draw[thickline] (n22) -- (n23);
			\draw[thickline] (n23) -- (n24);
			\draw[thickline,lightgray] (n24) -- (n25);
			\draw[thickline,lightgray] (n25) -- (n26);
			
			\draw[thickline] (n33) -- (n34);
			\draw[thickline] (n34) -- (n35);
			\draw[thickline,lightgray] (n35) -- (n36);
			
			\draw[thickline] (n44) -- (n45);
			\draw[thickline] (n45) -- (n46);
			
			\draw[thickline] (n55) -- (n56);
			
			\draw[thickline] (n12) -- (n22);
			\draw[thickline] (n13) -- (n23);
			\draw[thickline,lightgray] (n14) -- (n24);
			\draw[thickline,lightgray] (n15) -- (n25);
			\draw[thickline] (n16) -- (n26);
			
			\draw[thickline] (n23) -- (n33);
			\draw[thickline] (n24) -- (n34);
			\draw[thickline,lightgray] (n25) -- (n35);
			\draw[thickline,lightgray] (n26) -- (n36);
			
			\draw[thickline] (n34) -- (n44);
			\draw[thickline] (n35) -- (n45);
			\draw[thickline,lightgray] (n36) -- (n46);
			
			\draw[thickline] (n45) -- (n55);
			\draw[thickline] (n46) -- (n56);
			
			\draw[thickline] (n56) -- (n66);
			
			\draw[thickline]
			($(1,-0.5)$) -- ++(-0.5,0) -- ++(0,-6.1) -- ++(0.5,0);
			\draw[thickline]
			($(6,-0.5)$) -- ++(0.5,0) -- ++(0,-6.1) -- ++(-0.5,0);
			
		\end{tikzpicture}
		\caption{$\DC_r(\Gamma)$ at $ r = 1$}
	\end{subfigure}\hfill
	\begin{subfigure}[b]{0.32\textwidth}
		\centering
		\begin{tikzpicture}[scale=.6,
			pt/.style={circle,fill=black,inner sep=2pt},
			edge/.style={thick},
			]
			\begin{scope}
				\path (-3,-3) rectangle (3,3);
				
				\foreach \k in {0,...,5} {
					\coordinate (P\k) at ({1.7*cos(60*\k)},{1.7*sin(60*\k)});
				}
				
				\fill[gray!60,opacity=0.35] (P0)--(P1)--(P2)--cycle;
				\fill[gray!60,opacity=0.35] (P1)--(P2)--(P3)--cycle;
				\fill[gray!60,opacity=0.35] (P2)--(P3)--(P4)--cycle;
				\fill[gray!60,opacity=0.35] (P3)--(P4)--(P5)--cycle;
				\fill[gray!60,opacity=0.35] (P4)--(P5)--(P0)--cycle;
				\fill[gray!60,opacity=0.35] (P5)--(P0)--(P1)--cycle;
				
				\draw[edge] (P0)--(P1);
				\draw[edge] (P1)--(P2);
				\draw[edge] (P2)--(P3);
				\draw[edge] (P3)--(P4);
				\draw[edge] (P4)--(P5);
				\draw[edge] (P5)--(P0);
				
				\draw[edge] (P0)--(P2);
				\draw[edge] (P1)--(P3);
				\draw[edge] (P2)--(P4);
				\draw[edge] (P3)--(P5);
				\draw[edge] (P4)--(P0);
				\draw[edge] (P5)--(P1);
				
				\foreach \k in {0,...,5} {
					\node[pt] at (P\k) {};
				}
				
				\node[right]       at (P0) {\footnotesize $1$};
				\node[above right] at (P1) {\footnotesize $2$};
				\node[above left]  at (P2) {\footnotesize $3$};
				\node[left]        at (P3) {\footnotesize $4$};
				\node[below left]  at (P4) {\footnotesize $5$};
				\node[below right] at (P5) {\footnotesize $6$};
			\end{scope}
		\end{tikzpicture}
		\caption{$\Cech_r(\Gamma)$ at $r = 1$.}
	\end{subfigure}
	\caption{Time series with distance complex and \v{C}ech complex.
		\textbf{(a)} $ \Gamma = (x_1,\dots, x_6)$ consists of six points on the unit circle. 
		Note that in panel \textbf{(b)} the y-axis is reversed; this way, the vertices of $\DC(\Gamma)$ are at the respective matrix positions. 
		\textbf{(c)} The \v{C}ech complex is built on the time-index set of $\Gamma$; its simplices record nonempty intersections of the corresponding balls in state space. }
	\label{fig:dist-complex-example}
\end{figure}
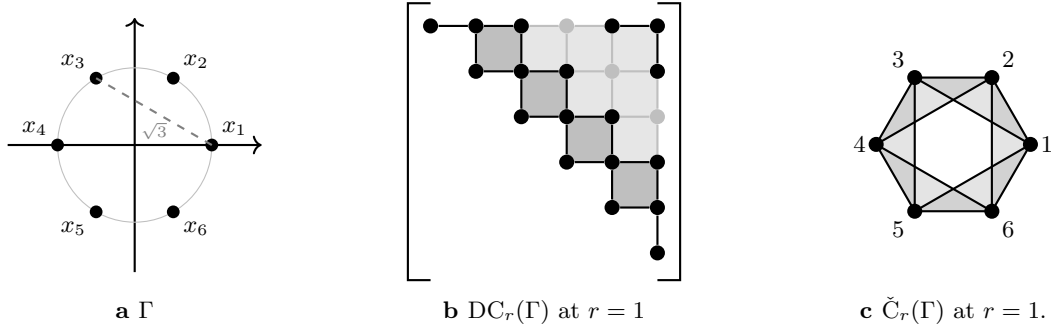

\subsection{Definition of the Distance-to-\v{C}ech Map}

For $r>R_\Gamma$, we define a degree-one homomorphism
\begin{equation}
	\psi_r \colon C_\bullet(\DC_r(\Gamma))\rightarrow C_{\bullet+1}(\Cech_r(\Gamma)).
	\label{eq:phi-dg-to-vr}
\end{equation}
For any vertex $ v = (i,j)$ in $ \DC_r(\Gamma) $, we set 
\begin{equation}
	\psi_r(v) = c_\Gamma(i,j) - [i,j].
	\label{eq-phi-dg-to-vr-0}
\end{equation}
For any edge $e $ in $ \DC_r(\Gamma) $, we set
\begin{equation}
	\psi_r(e) = \begin{cases}
		[i,j,i+1], & \text{ if } e = [i,i+1]\times[j,j], \\
		[i,j,j+1], & \text{ if } e = [i,i]\times[j,j+1], \\
	\end{cases}.
	\label{eq-phi-dg-to-vr-1}
\end{equation}
For any square $ q = [i,i+1]\times [j,j+1] \in \DC_r(\Gamma) $, we set
\begin{equation}
	\psi_r(q) = [i,i+1,j,j+1].
	\label{eq-phi-dg-to-vr-2}
\end{equation}
Here $ [\dots] $ denotes an oriented simplex, and any degenerate simplex appearing in these formulas is interpreted as zero.
We then extend $ \psi_r $ via linearity. 

\begin{defn}\label{def:distance-to-cech-map}
	We call $\psi$ the \emph{distance-to-\v{C}ech map} of $\Gamma$.
\end{defn}

The formulas on edges and squares are chosen so that $\psi$ commutes with the boundary operator, as illustrated in \cref{fig:chain-map-discrete}.

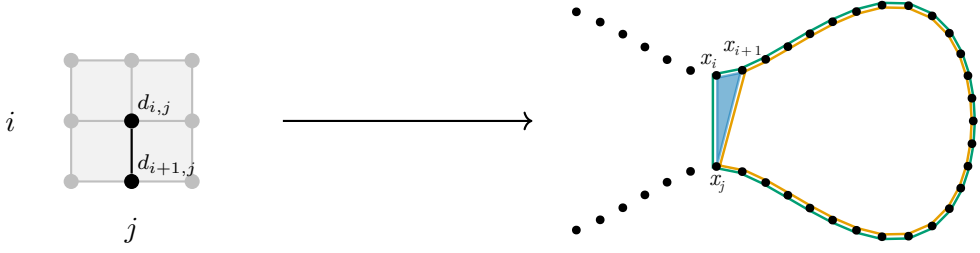
\begin{figure}
	\centering
	\begin{tikzpicture}[scale=.8,
		defaultnode/.style={circle, fill=lightgray, inner sep=2pt},
		specialnode/.style={circle, fill=black, inner sep=2pt},
		textnode/.style={inner sep=0pt, outer sep=0pt, font=\scriptsize}
		]
		
		\def\xoffset{1}
		\def\yoffset{1}
		
		\foreach \i in {1,2,3} {
			\foreach \j in {1,2,3} {
				\pgfmathtruncatemacro{\x}{\j + \xoffset}
				\pgfmathtruncatemacro{\y}{\i + \yoffset}
				\coordinate (c\i\j) at (\x, -\y);
			}
		}
		
		\foreach \i in {1,2} {
			\foreach \j in {1,2} {
				\path[fill=lightgray!20, draw=none]
				(c\i\j) -- (c\i\the\numexpr\j+1\relax) --
				(c\the\numexpr\i+1\relax\the\numexpr\j+1\relax) --
				(c\the\numexpr\i+1\relax\j) -- cycle;
			}
		}
		
		\foreach \i in {1,2,3} {
			\foreach \j in {1,2} {
				\draw[lightgray,thick] (c\i\j) -- (c\i\the\numexpr\j+1\relax);
			}
		}
		\foreach \j in {1,2,3} {
			\foreach \i in {1,2} {
				\draw[lightgray,thick] (c\i\j) -- (c\the\numexpr\i+1\relax\j);
			}
		}
		
		\foreach \i in {1,2,3} {
			\foreach \j in {1,2,3} {
				\ifnum\i=2\relax
				\ifnum\j=2
				\node[specialnode] (n\i\j) at (c\i\j) {};
				\else
				\node[defaultnode] (n\i\j) at (c\i\j) {};
				\fi
				\else\ifnum\i=3\relax
				\ifnum\j=2
				\node[specialnode] (n\i\j) at (c\i\j) {};
				\else
				\node[defaultnode] (n\i\j) at (c\i\j) {};
				\fi
				\else
				\node[defaultnode] (n\i\j) at (c\i\j) {};
				\fi\fi
			}
		}
		
		\node[textnode, above right=0pt of n22] {$d_{i,j}$};
		\node[textnode, above right=0pt of n32] {$d_{i+1,j}$};
		
		\draw[thick] (n22) -- (n32);
		
		\node at (1.0, -\yoffset - 2) {$i$};
		
		\node at (\xoffset + 2, -4.8) {$j$};
		
		\node[anchor=west, xshift=2.5cm] at (6.5, -3) (codomain)
		{\includegraphics[width=6cm]{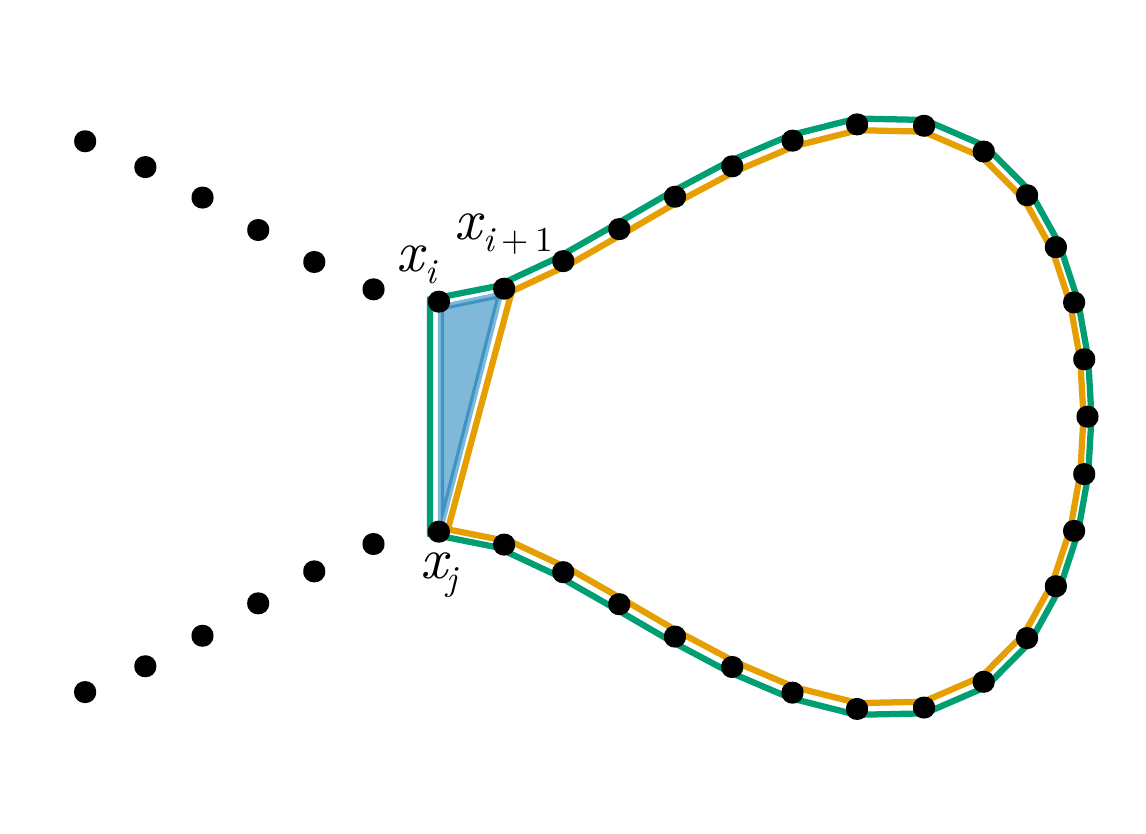}};
		
		\draw[->, thick] (5.5, -3) -- (codomain);
	\end{tikzpicture}
	\caption{Illustration of the chain map. The left panel is a snippet of the distance complex showing
		vertices $v_1 := d_{i,j}$ and $v_2 := d_{i+1,j}$, and an edge $e = [v_1,v_2]$.
		The vertices are mapped to the green and orange cycles, respectively.
		The edge is mapped to the triangle shown in blue.
		The picture shows that the chain map condition $ \partial \psi(e) = \psi(v_2) - \psi(v_1) $ is satisfied.
	}
	\label{fig:chain-map-discrete}
\end{figure}

\begin{lem}\label{lem:phi-dm-vr-well-def}
	$ \psi $ is a degree-one morphism of graded $(R_\Gamma,\infty)$-filtered modules.
\end{lem}
\begin{proof}
	For fixed $r>R_\Gamma$, \cref{eq:distance-complex} guarantees that the formulas in \cref{eq-phi-dg-to-vr-0,eq-phi-dg-to-vr-1,eq-phi-dg-to-vr-2}
	map generators of $ C_k(\DC_r(\Gamma)) $ into $ C_{k+1}(\Cech_r(\Gamma)) $.
	Thus, $\psi_r$ is a well-defined degree-one homomorphism of graded modules.
	Furthermore, the right-hand sides of \cref{eq-phi-dg-to-vr-0,eq-phi-dg-to-vr-1,eq-phi-dg-to-vr-2} are independent of $r$, which implies that $\psi_r$ is compatible with the filtration inclusions.
\end{proof}

Since $\psi$ is well-defined, we henceforth omit the subscript $r$ from $ \psi_r $.

\begin{lem}\label{lem:phi-dm-vr-chain-map}
	We have $ \psi\partial = \partial\psi $, so $\psi$ is a chain map.
\end{lem}
\begin{proof}
	It is enough to check the identity on elementary cubes.

	Any vertex $ v=(i,j)$ in $ \DC_r(\Gamma) $ trivially has $ \psi(\partial(v)) = 0 $. 
	We calculate
	\[
	\partial(\psi(v)) = \partial(c_\Gamma(i,j) ) - \partial( [i,j] )= ([j] - [i]) - ([j] - [i]) = 0,
	\]
	and thus, by linearity, $ \psi (\partial c) = \partial \psi(c) $ for all 0-chains.
	
	For an edge $e=[i,i+1]\times[j,j]$, with $v=(i,j)$ and $w=(i+1,j)$,
	\[
	\psi_r(\partial e)
	=
	\psi_r(w)-\psi_r(v)
	=
	c_\Gamma(i+1,j)-[i+1,j]
	-
	c_\Gamma(i,j)+[i,j].
	\]
	Since $c_\Gamma(i+1,j) - c_\Gamma(i,j)= - [i,i+1]$ and $ -[i+1,j] =[j,i+1]$, this becomes
	\[
	[j,i+1]-[i,i+1]+[i,j]
	=
	\partial[i,j,i+1]
	=
	\partial\psi_r(e).
	\]
	The case $e=[i,i]\times[j,j+1]$ is analogous.
	
	For a square $q=[i,i+1]\times[j,j+1]$, direct expansion gives
	\begin{align*}
		\psi(\partial q) &= \psi( [i+1]\times [j,j+1] - [i]\times [j,j+1] - [i,i+1]\times[j+1] + [i,i+1]\times[j])\\
		&= [i+1,j,j+1] - [i,j,j+1] - [i,j+1,i+1] + [i,j,i+1]\\
		&= [i+1,j,j+1] - [i,j,j+1] + [i,i+1,j+1] - [i,i+1,j]
	\end{align*}
	which equals
	\begin{align*}
		\partial (\psi(q)) &= \partial [i,i+1,j,j+1] \\
		&= [i+1,j,j+1] - [i,j,j+1] + [i,i+1,j+1]- [i,i+1,j].
	\end{align*}
	Since all other chain groups of $ \DC(\Gamma) $ are trivial, $\psi$ commutes with $\partial$.
\end{proof}

\begin{thm}
	For every $r>R_\Gamma$, the map $\psi_r$ is a well-defined degree-one chain map
	\[
	\psi_r\colon C_\bullet(\DC_r(\Gamma))\to C_{\bullet+1}(\Cech_r(\Gamma)).
	\]
	The maps $\psi_r$ are compatible with the filtration inclusions.
	Hence they define a degree-one morphism of $(R_\Gamma,\infty)$-filtered chain complexes
	\[
	\psi\colon C_\bullet(\DC(\Gamma))\to C_{\bullet+1}(\Cech(\Gamma)).
	\]
\end{thm}
\begin{proof}
	This follows from the previous two lemmas.
\end{proof}

\begin{exmp}
	Return to the hexagon in \cref{fig:dist-complex-example}, and again take radius $r = 1$.
	We compute the image of some vertices of the distance complex under the distance-to-\v{C}ech map.
	\begin{itemize}
		\item Take $u = (1,2)$. Then $\psi(u) = c_\Gamma(1,2) - [1,2] = 0$.
		\item Take $v = (1,3)$. Then 
		\[
		\psi(v) = c_\Gamma(1,3) - [1,3]  = [1,2] + [2,3] - [1,3].
		\]
		Since $u$ and $v$ are in the same connected component of the distance complex, we have $[\psi(v)] = [\psi(u)] = 0$ in homology.
		\item Take $w = (1,5)$. Then 
		\[
		\psi(w) = c_\Gamma(1,5) - [1,5] = [1,2] + [2,3] + [3,4] + [4,5] - [1,5].
		\] 
		This is a nontrivial cycle in the \v{C}ech complex. 
	\end{itemize}
\end{exmp}

\subsection{Stability of the Distance-to-\v{C}ech Map}

We formulate stability with respect to perturbations that preserve the
pointwise correspondence and temporal order of an ordered point cloud.

First, note the following for ordered \v{C}ech complexes.

\begin{lem}\label{lem:cech-stability}
	Let $\Gamma_1$ and $\Gamma_2$ be ordered point clouds of length $n$ in a
	metric space $(X,d)$. Let
	$\varepsilon=\max_{1\leq i\leq n}d(\Gamma_1(i),\Gamma_2(i))$.
	Then
	\[
		\Cech_r(\Gamma_1)\subseteq\Cech_{r+\varepsilon}(\Gamma_2)
		\qquad\text{and}\qquad
		\Cech_r(\Gamma_2)\subseteq\Cech_{r+\varepsilon}(\Gamma_1).
	\]
\end{lem}
\begin{proof}
	If $\sigma\in\Cech_r(\Gamma_1)$, choose
	$y\in\bigcap_{i\in\sigma}B(\Gamma_1(i),r)$.
	For every $i\in\sigma$, the triangle inequality gives
	$d(y,\Gamma_2(i))<r+\varepsilon$, so
	$\sigma\in\Cech_{r+\varepsilon}(\Gamma_2)$.
	The other inclusion follows by interchanging $\Gamma_1$ and $\Gamma_2$.
\end{proof}

\begin{thm}[Stability]
\label{thm:distance-to-cech-stability}
	Let $\Gamma=(x_1,\dots,x_n)$ and
	$\Gamma'=(x'_1,\dots,x'_n)$ be ordered point clouds in the same metric
	space $(X,d)$, and set
	\[
		\varepsilon=\max_{1\leq i\leq n}d(x_i,x'_i),
		\qquad
		R=\max\{R_\Gamma,R_{\Gamma'}\}.
	\]
	For every $r>R$, the inclusions induced by the identity on the index sets
	fit into the commuting diagrams
	\[
		\begin{tikzcd}[column sep=large,row sep=large]
			C_\bullet(\DC_r(\Gamma))
				\arrow[r,"\psi_{\Gamma,r}"]
				\arrow[d,hook]
			&
			C_{\bullet+1}(\Cech_r(\Gamma))
				\arrow[d,hook]
			\\
			C_\bullet(\DC_{r+\varepsilon}(\Gamma'))
				\arrow[r,"\psi_{\Gamma',r+\varepsilon}"']
			&
			C_{\bullet+1}(\Cech_{r+\varepsilon}(\Gamma'))
		\end{tikzcd}
	\]
	and
	\[
		\begin{tikzcd}[column sep=large,row sep=large]
			C_\bullet(\DC_r(\Gamma'))
				\arrow[r,"\psi_{\Gamma',r}"]
				\arrow[d,hook]
			&
			C_{\bullet+1}(\Cech_r(\Gamma'))
				\arrow[d,hook]
			\\
			C_\bullet(\DC_{r+\varepsilon}(\Gamma))
				\arrow[r,"\psi_{\Gamma,r+\varepsilon}"']
			&
			C_{\bullet+1}(\Cech_{r+\varepsilon}(\Gamma)).
		\end{tikzcd}
	\]
\end{thm}
\begin{proof}
	By \cref{lem:distance-complex-stability} and
	\cref{lem:cech-stability}, all vertical maps are well-defined inclusions.
	On each elementary cube, the two composites in the first diagram are given by the same formula from
	\cref{eq-phi-dg-to-vr-0,eq-phi-dg-to-vr-1,eq-phi-dg-to-vr-2}.
	Thus it is easy to see that the diagrams commute.
\end{proof}

\begin{defn}
	Let $M=(M_r)_{r>R}$ and $N=(N_r)_{r>R}$ be filtered modules or filtered chain complexes, with structure maps denoted by $\iota^M_{r,s}$ and $\iota^N_{r,s}$.
	An \emph{$\varepsilon$-interleaving} between $M$ and $N$ consists of filtration-compatible maps
	\[
		f_r\colon M_r\to N_{r+\varepsilon},
		\qquad
		g_r\colon N_r\to M_{r+\varepsilon}
	\]
	such that
	\[
		g_{r+\varepsilon}f_r=\iota^M_{r,r+2\varepsilon},
		\qquad
		f_{r+\varepsilon}g_r=\iota^N_{r,r+2\varepsilon}.
	\]

	Let $\varphi\colon M\to P$ and $\varphi'\colon N\to Q$ be morphisms of
	filtered modules or filtered chain complexes.
	We call them \emph{$\varepsilon$-interleaved} if there are
	$\varepsilon$-interleavings $(f,g)$ between $M$ and $N$ and $(f',g')$
	between $P$ and $Q$ such that, for every $r>R$, the following diagrams commute:
	\[
		\begin{tikzcd}[column sep=large,row sep=large]
			M_r
				\arrow[r,"\varphi_r"]
				\arrow[d,"f_r"']
			&
			P_r
				\arrow[d,"f'_r"]
			\\
			N_{r+\varepsilon}
				\arrow[r,"\varphi'_{r+\varepsilon}"']
			&
			Q_{r+\varepsilon}
		\end{tikzcd}
		\qquad
		\begin{tikzcd}[column sep=large,row sep=large]
			N_r
				\arrow[r,"\varphi'_r"]
				\arrow[d,"g_r"']
			&
			Q_r
				\arrow[d,"g'_r"]
			\\
			M_{r+\varepsilon}
				\arrow[r,"\varphi_{r+\varepsilon}"']
			&
			P_{r+\varepsilon}.
		\end{tikzcd}
	\]
\end{defn}

\begin{cor}\label{cor:distance-to-cech-stability}
	Under the hypotheses of \cref{thm:distance-to-cech-stability}, the persistence morphisms
	\[
		H_k(\psi_\Gamma)\colon
		H_k(\DC(\Gamma))\longrightarrow H_{k+1}(\Cech(\Gamma))
	\]
	and $H_k(\psi_{\Gamma'})$ are $\varepsilon$-interleaved for every $k$.
	Consequently, their kernel, image, and cokernel persistence modules are $\varepsilon$-interleaved.
\end{cor}
\begin{proof}
	The first part follows directly from the theorem by applying homology to the diagrams.
	It remains only to note that an interleaving of morphisms as defined above induces interleavings of their kernels, images, and cokernels.
\end{proof}

The last statement gives stability for the barcodes of the kernels, images, and cokernels of the distance-to-\v{C}ech maps. 
More precisely, let $d_B$ denote the bottleneck distance between persistence diagrams, and let the supremum distance between ordered point clouds be
$
	d_\infty(\Gamma,\Gamma')
	=
	\sup_{1\leq i\leq n}d(x_i,x'_i).
$
With field coefficients, the distance-complex and \v{C}ech filtrations of ordered point clouds have pointwise finite-dimensional homology groups.
By \cref{cor:distance-to-cech-stability}, for $\mathcal F\in\{\ker,\operatorname{im},\operatorname{coker}\}$, the corresponding persistence diagrams satisfy
\[
	d_B\!\left(
		\operatorname{Dgm}\mathcal F(H_k(\psi_\Gamma)),
		\operatorname{Dgm}\mathcal F(H_k(\psi_{\Gamma'}))
	\right)
	\leq d_\infty(\Gamma,\Gamma').
\]

\subsection{The Distance-to-Rips Map}

Given an ordered point cloud $\Gamma=(x_1,\dots,x_n)$, we have
$ \Cech_r(\Gamma) \subset \VR_{2r}(\Gamma) $ for every $r\geq 0$.
We thus get a map from the distance filtration to the Vietoris--Rips filtration by composing the distance-to-\v{C}ech map with this inclusion.

A different (and better-behaved; see \cref{rem:distance-to-rips-vs-inclusion}) map into the Vietoris--Rips filtration can be obtained via the Kuratowski embedding \cite{Kuratowski1935}.
We now introduce some background, following Section 2.2 of \cite{LimMemoliOkutan2024}, where the Kuratowski embedding was studied in detail.

\begin{defn}
	A metric space $ (Y,d_Y) $ is called hyperconvex if 
	\begin{enumerate}[label=(\roman*)]
		\item it is \emph{metrically convex}: for all $x,y\in Y$, $r_1,r_2\geq 0$
		\[
			d(x,y) \leq r_1 + r_2 \iff \overline{B(x,r_1)} \cap \overline{B(y,r_2)} \neq \emptyset,
		\]
		\item it has the \emph{binary Helly property}: any pairwise intersecting family of closed balls has a common point.
	\end{enumerate}	
\end{defn}
The binary Helly property also holds for finite collections of open balls:

\begin{lem}\label{lem:hyperconvexity-open-ball-intersection}
	Let $(Y,d_Y)$ be hyperconvex, and let $x_i \in Y$ and $r_i\geq 0$ for $i\in I = \{ 1,\dots, n\}$.
	If $ B(x_i,r_i) \cap B(x_j,r_j)\neq \emptyset$ for all $ i,j\in I $, then
	\[
		\bigcap_{i\in I} B(x_i,r_i) \neq \emptyset.
	\]
\end{lem}
\begin{proof}
	Since there are only finitely many balls, there is $ s > 0 $ such that
	\[
	\overline{B(x_i,r_i-s)} \cap \overline{B(x_j,r_j-s)} \neq \emptyset \quad\text{for all }i,j\in I.
	\] 
	By hyperconvexity of $Y$, the family of closed balls $ \{\overline{B(x_i,r_i-s)} \}_{i\in I} $ has a common point. This point is also in the intersection of the open balls $B(x_i,r_i)$.
\end{proof}

Hyperconvex metric spaces are also called injective, in reference to the category-theoretic notion; see, e.g., the introduction of \cite{Lang2013InjectiveHulls}.

\begin{exmp}
	For any set $S$, the Banach space $L^\infty(S)$ of all bounded real-valued functions on $S$ with the $\ell^\infty$ norm is hyperconvex. For a proof, see \cite{Lang2013InjectiveHulls}.
\end{exmp}

\begin{defn}
	The \emph{Kuratowski embedding} of a compact metric space $(Y,d_Y)$ is the map
	\[
		\kappa\colon Y \rightarrow L^\infty(Y), \quad y\mapsto d_Y(y,\cdot).
	\]
\end{defn}

It is not hard to show 
\begin{equation}\label{eq:d-Y-norm-kappa}
	d_Y(y,y') = \norm{ \kappa(y) - \kappa(y')}_\infty,
\end{equation}
which implies that $\kappa$ is an isometric embedding. 
The key point for our purposes is that, for subsets of a hyperconvex metric space, the Vietoris--Rips complex is the \v{C}ech complex at half the radius.

\begin{thm}[cf.\ Proposition 2.2 in \cite{LimMemoliOkutan2024}]\label{thm:cech-rips-hyperconvex}
	Let $A $ be a nonempty subset of a hyperconvex metric space $(Y,d_Y)$. Then for any $r\geq 0$, 
	\[
		\Cech_r(A) = \VR_{2r}(A).
	\]
\end{thm}
\begin{proof}
	The inclusion $ \subset $ is well-known.
	To show equality, fix a simplex $\sigma$ in $\VR_{2r}(A)$. 
	By definition, all its edges have length less than $2r$, and thus
	\[
		B(x, r) \cap B(y,r) \neq \emptyset\quad \text{for all }x,y\in \sigma.
	\]
	A simplex has only finitely many vertices, so \cref{lem:hyperconvexity-open-ball-intersection} implies $\bigcap_{x\in\sigma } B(x,r) \neq \emptyset$. Therefore $\sigma$ is in $\Cech_r(A)$.
\end{proof}

We now define the Vietoris--Rips version of the distance complex.

\begin{defn}\label{def:rips-distance-complex}
	The \emph{(Vietoris--)Rips distance complex} of $\Gamma$ is
	\[
		\RDC(\Gamma)=\bigl(\RDC_r(\Gamma)\bigr)_{r\geq 0},
	\]
	where $\RDC_r(\Gamma)$ is the full cubical subcomplex of $\Cub(\mathbb R^2)$
	induced by the vertex set
	\[
		V_r=\{(i,j)\mid 1\leq i\leq j\leq n,\ d(x_i,x_j) < r\}.
	\]
\end{defn}

Given a finite ordered point cloud $\Gamma=(x_1,\dots,x_n)$, we now relate the Rips distance complex to the distance complex of its Kuratowski embedding $\kappa(\Gamma)=(\kappa(x_1),\dots,\kappa(x_n))$.
For this, we need the metric step size of $\Gamma$,
\[
	S_\Gamma = \max_{2\leq i\leq n}d(x_{i-1},x_i).
\]
Clearly, $S_\Gamma/2\leq R_\Gamma\leq S_\Gamma$; in Euclidean spaces or, more generally, metrically convex spaces, we have $R_\Gamma=S_\Gamma/2$.

\begin{prop}\label{prop:rdc-is-dc}
	Provided $r>S_\Gamma$, we have
	\[
		\RDC_r(\Gamma) = \DC_{r/2}(\kappa(\Gamma)).
	\]
\end{prop}
\begin{proof}
	Using \cref{thm:cech-rips-hyperconvex},
	we have
	\[
		\DC_{r/2}(\kappa(\Gamma)) = \left\{ [i,j]\times [k,l] \in \Cub([1,n]^2)\mid j\leq k\text{ and }\{ i,j,k,l\} \in \VR_r(\Gamma) \right\}.
	\]
	Therefore, the vertex set $V_r$ from \cref{def:rips-distance-complex} is the vertex set of $ \DC_{r/2}(\kappa(\Gamma)) $. 

	It is left to show that $ \DC_{r/2}(\kappa(\Gamma)) $ is the full subcomplex of $ \Cub([1,n]^2) $ with vertex set $V_r$.
	Suppose all vertices of an elementary cube $Q = [i,j] \times [k,l]$ are contained in $\DC_{r/2}(\kappa(\Gamma))$.
	Then, the edges $ \{i,k\}, \{ i,l \}, \{j,k\}, \{ j,l \} $ are in $\VR_r(\Gamma) $ by the formula above for the distance complex.
	Furthermore, the condition $r>S_\Gamma$ ensures $ d(x_i,x_j) < r$ and $d(x_k,x_l)<r$, which yields $ \{i,j\}, \{ k,l \} \in \VR_r(\Gamma)$.
	Thus every two-element subset of $\{i,j,k,l\}$ belongs to $\VR_r(\Gamma)$.
	By the clique property, $\{i,j,k,l\}\in\VR_r(\Gamma)$, and hence
	$Q\in\DC_{r/2}(\kappa(\Gamma))$.
	Therefore, $ \DC_{r/2}(\kappa(\Gamma)) $ is the full subcomplex of $ \Cub([1,n]^2) $ with vertex set $V_r$.
\end{proof}

For $ r\leq S_\Gamma $, the Rips distance complex need not agree with the distance complex of the Kuratowski embedding. 
We chose this definition of the Rips distance complex since it is easier to compute and more directly related to the distance matrix of the time series.

\begin{defn}
	The \emph{distance-to-Rips map} of $ \Gamma $ is the distance-to-\v{C}ech map of $ \kappa(\Gamma) $,
	which we denote by
	\[
		\tilde\psi\colon \RDC(\Gamma) \rightarrow \VR(\Gamma).
	\]
\end{defn}

Since $L^\infty(\Gamma)$ is metrically convex,
$R_{\kappa(\Gamma)}=S_\Gamma/2$. Thus, under the identifications
$\RDC_r(\Gamma)=\DC_{r/2}(\kappa(\Gamma))$ and
$\VR_r(\Gamma)=\Cech_{r/2}(\kappa(\Gamma))$, the distance-to-\v{C}ech map at
parameter $r/2$ defines the distance-to-Rips map at parameter $r$ for every
$r>S_\Gamma$.

\begin{exmp}\label{ex:close-return-dm-and-rips}
	Let $\Gamma=(x_1,\dots,x_{41})$ be the time series shown in \cref{fig:chain-map-example-ts}.
	It has a close return between $x_7$ and $x_{35}$, producing an off-diagonal low-value region in the distance matrix shown in \cref{fig:chain-map-example-dm}.
	In \cref{fig:chain-map-example-levels}, we show selected sublevel sets of the distance matrix.
	\begin{itemize}
		\item At the first filtration value, only the diagonal component is present in the distance complex.
		\item At the second filtration value, the off-diagonal component appears in the Rips distance complex. At the same time, the Vietoris--Rips complex has a nontrivial 1-cycle, which is the image of the off-diagonal component under the distance-to-Rips map.
		\item At the third filtration value, the off-diagonal component is still present, but the 1-cycle is filled in the Vietoris--Rips complex.
		\item At the fourth filtration value, the off-diagonal component has merged with the diagonal component. A nontrivial 1-cycle is present in the Rips distance complex; however, its image under the distance-to-Rips map is trivial in homology. This is not obvious from the picture, but a computation shows that the degree-two homology of the Vietoris--Rips complex is trivial.
	\end{itemize}
\end{exmp}

\begin{rem}\label{rem:comparison-with-ichinomiya}
	We now describe the relationship between our distance complex and the filtration used by Ichinomiya \cite{Ichinomiya2025}.

	Any matrix can be regarded as an image by assigning a color (for example, a grayscale value) to each matrix entry.
	There are two common filtrations associated with images (see, e.g., \cite{Bleile2022,Robins2011,stucki23a}):
	the $ \mathcal{V} $-construction and the $\mathcal{T}$-construction; their relationship is discussed in \cite{Bleile2022}.

	The $ \mathcal{V} $-construction of $M\in \R^{d\times d}$ is a filtered cubical complex which, at filtration value $r$, is the full subcomplex of $ \Cub([1,d]^2) $ that has all vertices $(i,j)$ where $ m_{i,j}< r $.
	The $\mathcal{T}$-construction proceeds the other way around: At filtration value $r$, one takes the subcomplex of $ \Cub([0,d]^2) $ that is generated by all squares $ [i-1,i] \times [j-1,j]$ where $m_{i,j}<r$. 
	
	Given a time-ordered distance matrix of a time series, the Rips distance complex is the upper right part of the $ \mathcal{V} $-construction associated with this matrix, whereas Ichinomiya \cite{Ichinomiya2025} considers the $\mathcal{T}$-construction of the same matrix.
\end{rem}

\begin{figure}[ht]
	\centering
	\begin{subfigure}[t]{0.58\textwidth}
		\centering
		\includegraphics[width=.8\textwidth]{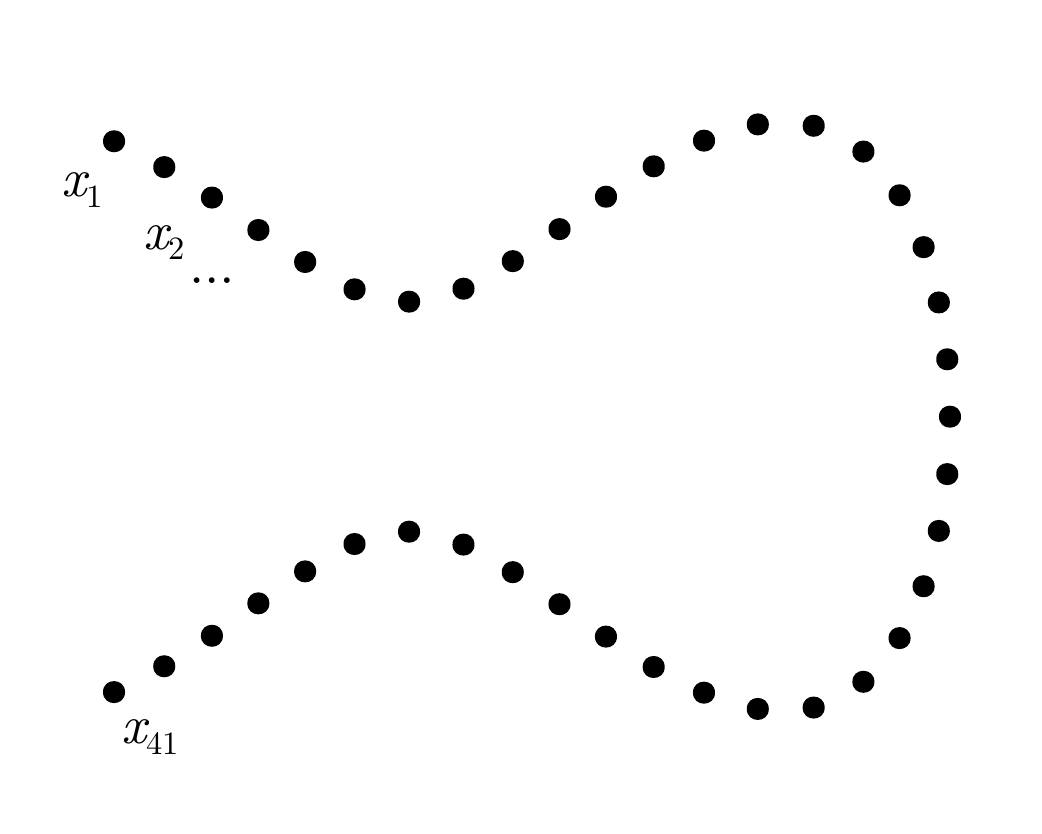}
		\caption{}
		\label{fig:chain-map-example-ts}
	\end{subfigure}
	\hfill
	\begin{subfigure}[t]{0.36\textwidth}
		\centering
		\includegraphics[width=.8\textwidth]{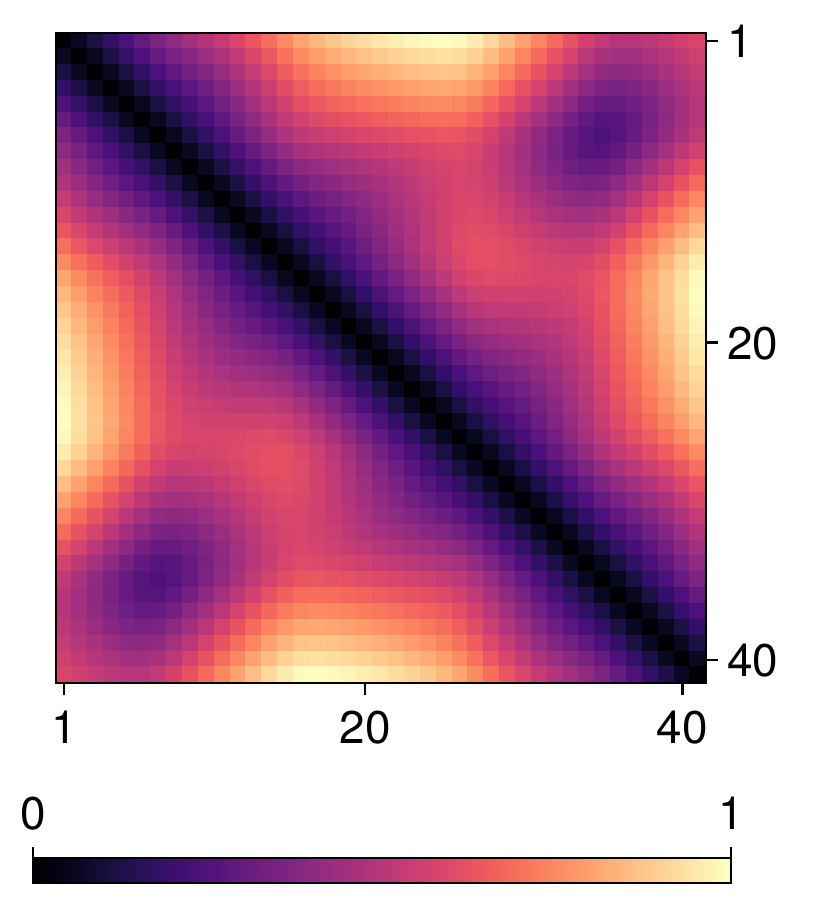}
		\caption{}
		\label{fig:chain-map-example-dm}
	\end{subfigure}
	
	\begin{subfigure}[t]{0.98\textwidth}
		\centering
		\includegraphics[width=.95\textwidth]{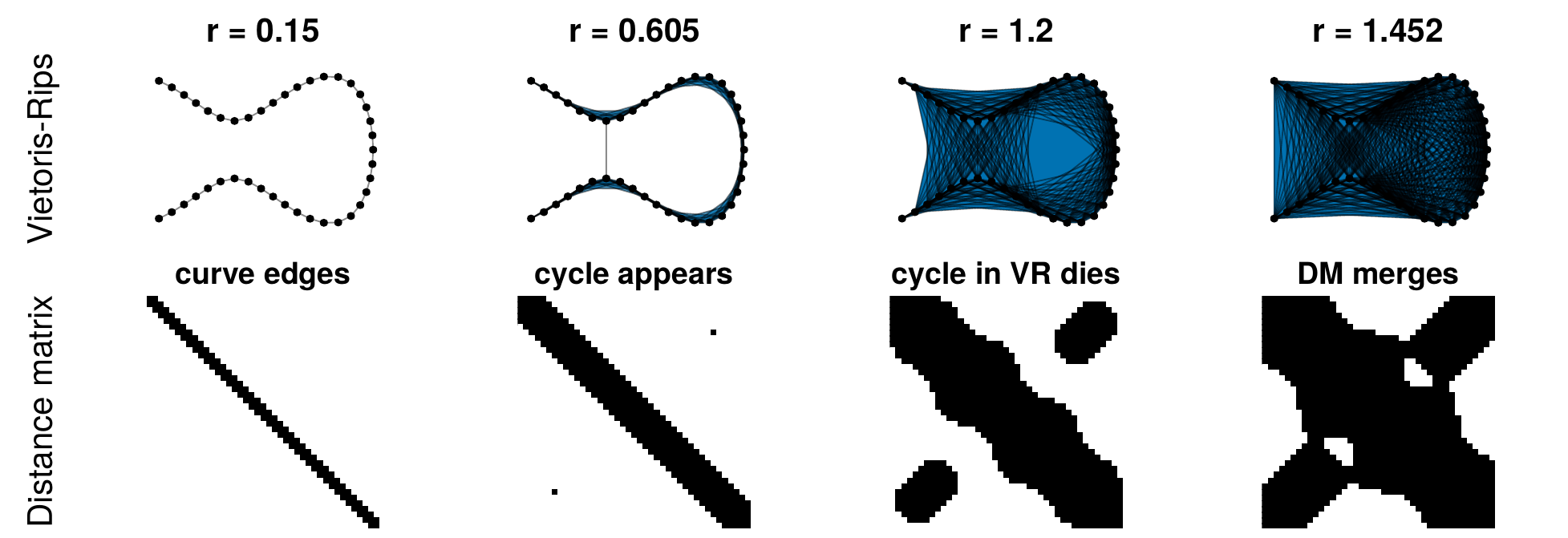}
		\caption{}
		\label{fig:chain-map-example-levels}
	\end{subfigure}
	\caption{A sampled curve with a close return.
		\textbf{(a)}~The time series $\Gamma=(x_1,\dots,x_{41})$ in state space.
		\textbf{(b)}~Heatmap of its distance matrix, with rows and columns ordered by time.
		Low values away from the diagonal mark pairs of distinct time indices whose states are close.
		\textbf{(c)}~Four filtration values shown both in the Vietoris--Rips filtration and in the binary sublevel sets of the distance matrix.
		At the second value, an off-diagonal component records the close return that gives a cycle under the distance-to-Rips map.
	}
	\label{fig:chain-map-example}
\end{figure}

\subsection{Properties of the Induced Map in Degree Zero}

We now consider the induced map $ H(\psi) $.
In this section, we consider homology with respect to an arbitrary coefficient ring. 
We show that the induced map from $H_0$ to $H_1$ is surjective and then discuss injectivity.

\begin{thm}\label{thm:distance-to-cech-surjective}
	For every $r>R_\Gamma$, the map $H_0(\psi_r)$ is surjective.
\end{thm}
\begin{proof}
	Let
	\[
		c = \sum_{i=1}^{m} \alpha_i [k_i,l_i] \in C_1(\Cech_r(\Gamma)),
	\]
	be a cycle.
	After changing the signs of the coefficients if necessary, we may assume $k_i<l_i$ for all $i$.
	Set $v_i=(k_i,l_i)$ and
	$
		v = -\sum_{i=1}^{m} \alpha_i v_i \in C_0(\DC_r(\Gamma)).
	$
	Since $ [k_i,l_i]\in \Cech_r(\Gamma) $, the vertex $(k_i,l_i)$ lies in $\DC_r(\Gamma)$, and therefore $v\in C_0(\DC_r(\Gamma))$. 
	
	Let
	\[
		b = \sum_{i=1}^{m} \alpha_i c_\Gamma(k_i,l_i).
	\]
	By definition of $\psi_r$, we have
	\[
		\psi_r(v)
		=
		-\sum_{i=1}^{m}\alpha_i\left(c_\Gamma(k_i,l_i)-[k_i,l_i]\right)
		=
		c-b.
	\]
	Since $v$ is a 0-cycle and $\psi_r$ is a chain map, $\partial\psi_r(v)=0$.
	Together with $\partial c=0$, this implies $\partial b=0$.
	But $b$ is supported on the subcomplex of $ \Cech(\Gamma) $ that is the path graph with vertices $1,\dots,n$ and edges $[j-1,j]$ for $j=2,\dots,n$. 
	By definition of the \v{C}ech complex of the indexed family $\Gamma$, these are distinct vertices; therefore, the graph has no nontrivial 1-cycles.
	Therefore $b=0$, and hence $\psi_r(v)=c$.
	This proves that $H_0(\psi_r)$ is surjective.
\end{proof}

\begin{rem}\label{rem:distance-to-rips-vs-inclusion}
	The preceding theorem, together with \cref{prop:rdc-is-dc}, implies that the distance-to-Rips map induces a surjective morphism in degree zero.
	This is not the case for the map induced by the composition $ \DC(\Gamma) \xrightarrow{\psi} \Cech(\Gamma)\hookrightarrow \VR_{2\cdot}(\Gamma)$. In this sense, the distance-to-Rips map is better-behaved.
\end{rem}

\subsubsection*{Injectivity}
An obvious obstruction to injectivity of the morphism $ H_0(\psi) $ is that the kernel always contains the connected component of the diagonal of the distance matrix.
This can be fixed by considering reduced homology via the augmentation map that sends the diagonal component to a generator of the $(-1)$-chains. 

However, even the map in reduced homology is generally not injective.
A curve cycle is null-homologous in the geometric complex if it is the sum of the boundaries of triangles.
The corresponding connected component in the distance complex is connected to the diagonal if the cycle is the sum of boundaries of triangles that have a curve edge as one side.
This is precisely what happens in \cref{ex:close-return-dm-and-rips} at the third filtration value: The off-diagonal component is still present, but the 1-cycle is filled in the Vietoris--Rips complex.

Injectivity also fails for a more fundamental reason: Different segments of the time series can move in parallel around the same geometric feature.
This is precisely the difference between spatial and temporal homology that we want to leverage for time series analysis. 

\begin{figure}
	\centering
	\begin{subfigure}[t]{0.32\textwidth}
		\centering
		\includegraphics[width=\textwidth]{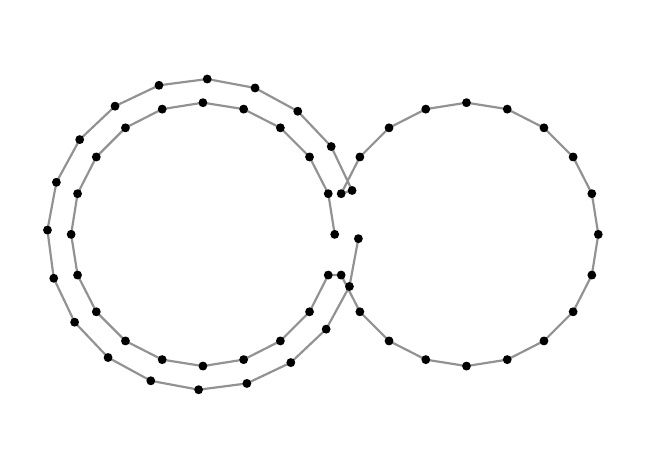}
		\caption{}
		\label{fig:two-left-turns-ts}
	\end{subfigure}
	\hfill
	\begin{subfigure}[t]{0.28\textwidth}
		\centering
		\includegraphics[width=\textwidth]{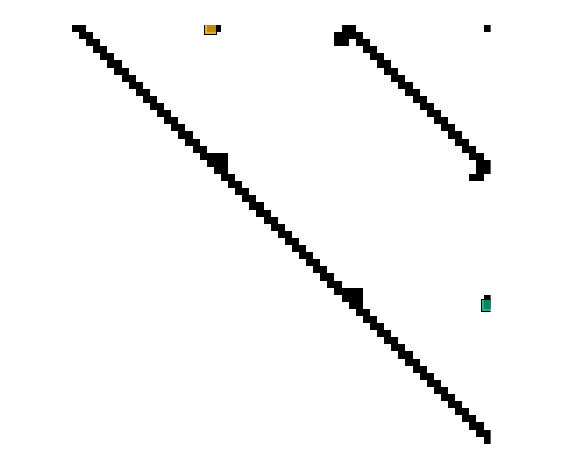}
		\caption{}
		\label{fig:two-left-turns-dm}
	\end{subfigure}
	\hfill
	\begin{subfigure}[t]{0.32\textwidth}
		\centering
		\includegraphics[width=\textwidth]{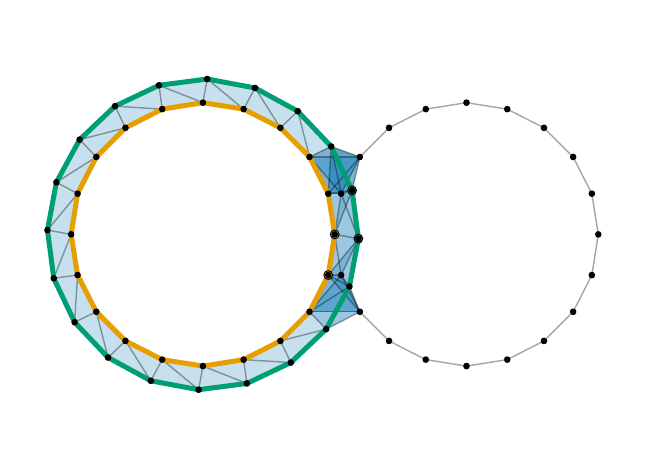}
		\caption{}
		\label{fig:two-left-turns-rips}
	\end{subfigure}
	\caption{
		An ordered point cloud with a left, right, left sequence of almost-closed loops.
		\textbf{a.} The sampled time series in $\R^2$.
		\textbf{b.} The corresponding sublevel set of the distance matrix at a radius larger than the maximum distance between consecutive samples and smaller than the distance between samples two steps apart.
		The orange and green pixels mark the close returns near matrix entries $(1,20)$ and $(41,60)$, respectively.
		\textbf{c.} The Vietoris--Rips complex at the same radius, with the induced cycles highlighted in the same colors.
	}
	\label{fig:two-left-turns-example}
\end{figure}

\begin{exmp}\label{ex:two-left-turns}
	Consider the ordered point cloud in \cref{fig:two-left-turns-example}.
	It visits the left hole twice, with a segment in between that moves around the right-hand hole. Each visit to the left hole creates an off-diagonal connected component in the Rips distance complex, visible in \cref{fig:two-left-turns-dm}.
	Under the distance-to-Rips map, these components map to homologous cycles in the Vietoris--Rips complex: both wind around the left hole.
\end{exmp}

\subsection{The Map in Degree One}

Since the distance complex is a finite cubical complex in $\R^2$, its homology vanishes in degrees greater than one.
Therefore, among the higher induced maps, only
\[
	H_1(\psi_r)\colon H_1(\DC_r(\Gamma))\to H_2(\Cech_r(\Gamma))
\]
can be nonzero.
The following examples indicate two phenomena: $H_1(\psi)$ can be nontrivial, but it is not generally surjective.
For simplicity, we consider the distance-to-Rips map $\tilde\psi$ in what follows.

\begin{figure}
	\centering
	\begin{subfigure}[t]{0.4\textwidth}
		\centering
		\includegraphics[width=.8\textwidth]{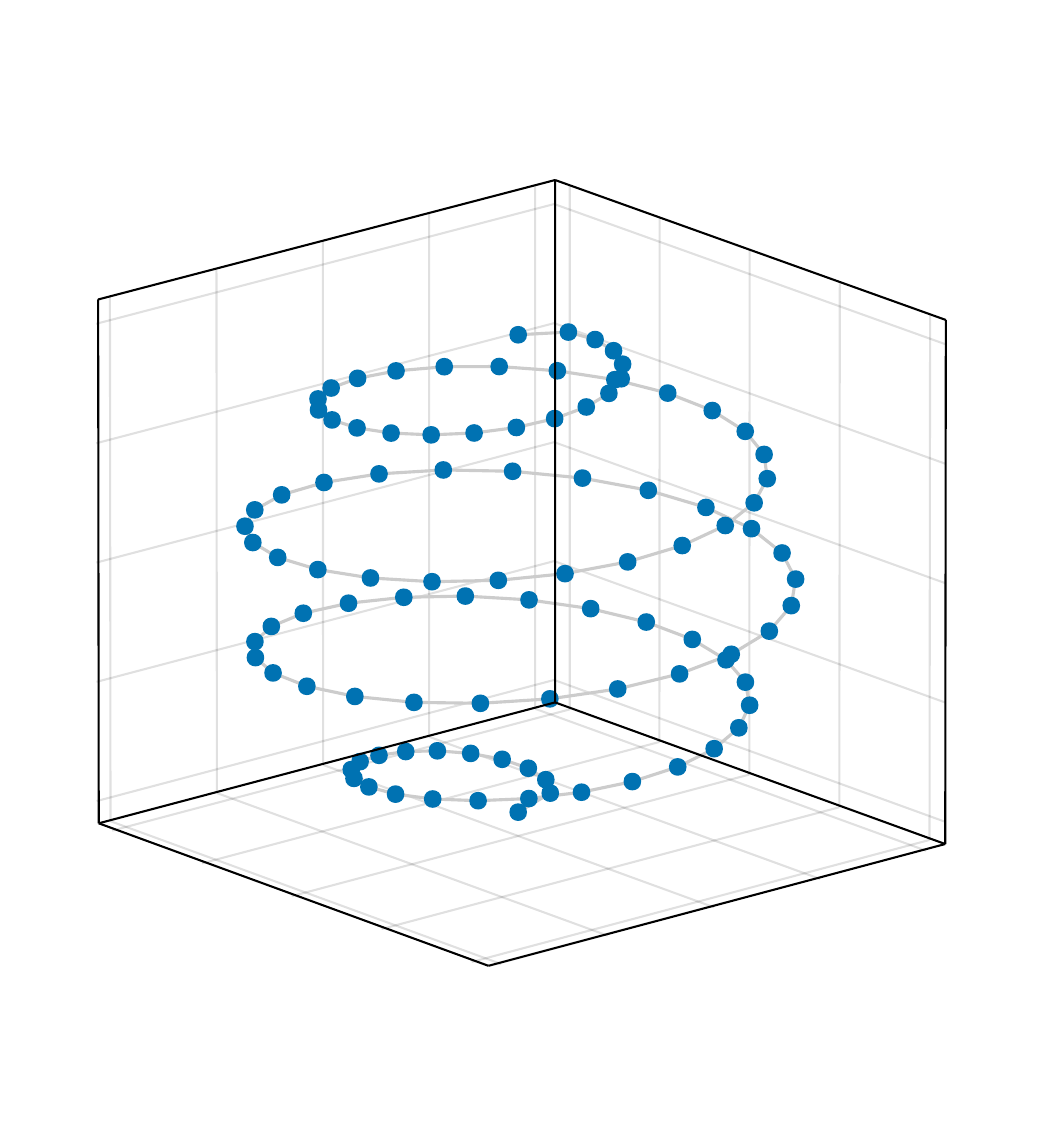}
		\caption{}
		\label{fig:s2-helix-ts}
	\end{subfigure}
	\begin{subfigure}[t]{0.28\textwidth}
		\centering
		\includegraphics[width=\textwidth]{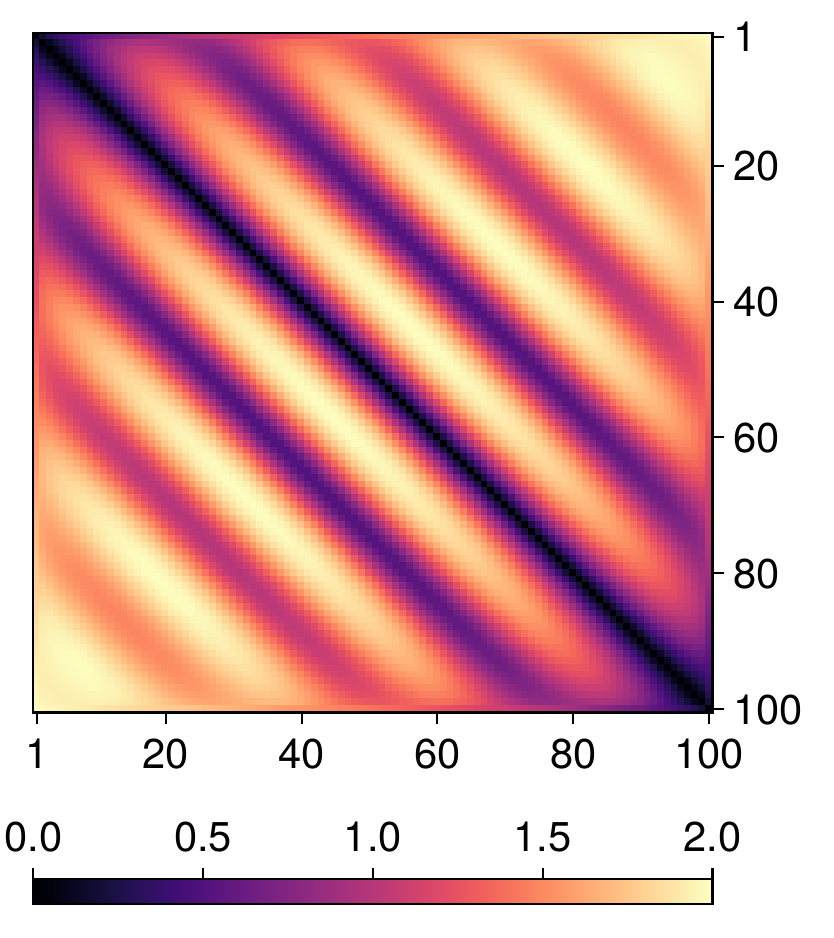}
		\caption{}
		\label{fig:s2-helix-dm}
	\end{subfigure}
	\hfill
	\begin{subfigure}[t]{0.28\textwidth}
		\centering
		\includegraphics[width=\textwidth]{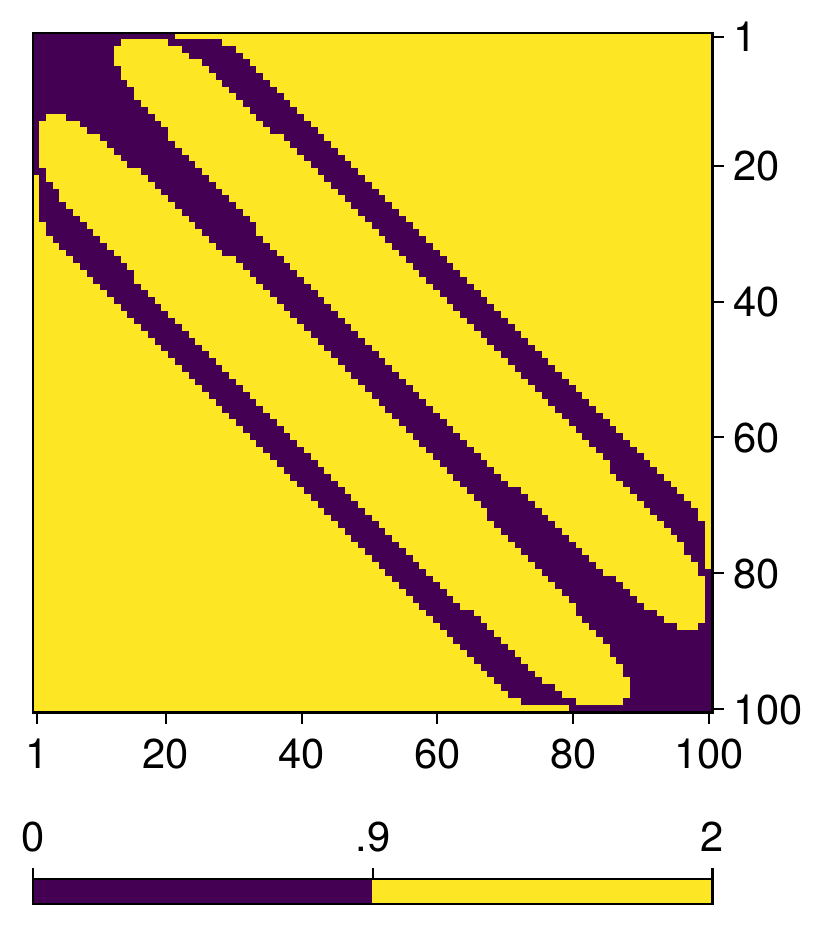}
		\caption{}
		\label{fig:s2-helix-dm-q}
	\end{subfigure}
	\caption{Spiral time series with distance matrix. 
		\textbf{a.} Time series sampled from a curve that spirals upward on the surface of the unit 2-sphere.
		\textbf{b.} Distance matrix as a heatmap. Note the ridges parallel to the diagonal.
		\textbf{c.} Sublevel set of \textbf{b} with nontrivial $H_1$.
	}
	\label{fig:s2-helix}
\end{figure}

\begin{exmp}[Nontrivial Induced Morphism]
	Consider the time series $\gamma$ in \cref{fig:s2-helix-ts}. 
	It lies on the unit 2-sphere centered at $(0,0,1)$ and spirals from the south pole to the north pole.
	Specifically, $\gamma=(\eta(t_k))_{k=1}^{100}$, where
	\[
		\eta(t)
		=
		\left(\sqrt{1-\left(\frac{t}{4\pi}-1\right)^2}\cos(t), 
			\sqrt{1-\left(\frac{t}{4\pi}-1\right)^2}\sin(t),
			\frac{t}{4\pi}
		\right),\qquad
		t_k=\frac{8\pi(k-1)}{99}.
	\]
	We now describe a computation that confirms $ \im H_1(\tilde\psi) \neq 0$.
	The code is available at \cite{paperGithubRepo}. We use Ripserer \cite{Ripserer2020} (a Julia implementation of the persistence software Ripser \cite{Bauer2021Ripser}) for persistent homology computations. In what follows, all computed values are rounded to three decimal places.
		
	The degree-two persistent homology of the embedded point cloud is
	\[
		H_2(\VR(\gamma)) \cong \mathbb{F}_{43}(0.899,1.650].
	\]
	For a field $\mathbb F$ and an interval $I\subset\R$, we write $\mathbb F I$ for the corresponding interval module.

	The distance matrix of $\gamma$ is visualized in \cref{fig:s2-helix-dm};
	the sublevel set in \cref{fig:s2-helix-dm-q} indicates nontrivial homology in degree one. 
	Indeed, we have
	\[
		H_1(\RDC(\gamma)) \cong \mathbb{F}_{43}(0.899,2.00] \oplus \mathbb{F}_{43}(1.35,2.00] \oplus \mathbb{F}_{43}(1.67,2.00] \oplus \mathbb{F}_{43}(1.94,2.00].
	\]

	We compute a cohomology generator $ \alpha $ corresponding to the bar in $ H_2(\VR(\gamma))$ and homology generators $v_1,v_2$ corresponding to the first two bars in $ H_1(\RDC(\gamma)) $. 
	The other two bars do not overlap with the $H_2$-bar, so we cannot evaluate the Kronecker pairing.
	The metric step size is $S_\gamma=0.254$, so $\tilde\psi_r$ is defined on the filtration values used here.
	A computation (using coefficients in $\F_{43}$) yields
	\[
		\alpha(\tilde\psi(v_1)) = -1,\quad \alpha(\tilde\psi(v_2)) = -1.
	\]
	If $ H_1(\tilde\psi) $ were trivial, all Kronecker pairings above would be zero.
	Therefore, $ H_1(\tilde\psi) $ is nontrivial. \qed
\end{exmp}

\begin{exmp}
	Let $\gamma = \{ x_k = (\cos(2\pi (k-1)/6), \sin(2\pi (k-1)/6)) \mid k = 1,\dots,6\}$ be the vertices of a regular hexagon.
	The distance matrix of $\gamma$ is
	\begin{center}
	\begin{tikzpicture}
			\begin{scope}[shift={(-3,0)}]
				\foreach \k in {1,...,6} {
					\coordinate (P\k) at ({1.25*cos(60*(\k-1))},{1.25*sin(60*(\k-1))});
					\fill (P\k) circle (2pt) node[above right] {\footnotesize $ x_{\k} $};
				}
			\end{scope}
			
			\begin{scope}[shift={(3,0)}]
				\node at (0,0) {$D = \begin{pmatrix}
						0 & 1 & \sqrt{3} & 2 & \sqrt{3} & 1 \\
						1 & 0 & 1 & \sqrt{3} & 2 & \sqrt{3} \\
						\sqrt{3} & 1 & 0 & 1 & \sqrt{3} & 2 \\
						2 & \sqrt{3} & 1 & 0 & 1 & \sqrt{3} \\
						\sqrt{3} & 2 & \sqrt{3} & 1 & 0 & 1 \\
						1 & \sqrt{3} & 2 & \sqrt{3} & 1 & 0
					\end{pmatrix}$};
			\end{scope}
		\end{tikzpicture}
	\end{center}
	A direct computation gives
	\[
	H_0(\RDC(\gamma);\mathbb F) \cong \mathbb F(0,\infty) \oplus \mathbb F(0,1]^{\oplus5} \oplus \mathbb F(1,2] \quad \text{ and }\quad H_1(\RDC(\gamma);\mathbb F) \cong 0.
	\]
	The map $\tilde\psi_r$ is defined for $r\geq \sqrt{3}$. 
	Furthermore, $H_1(\tilde\psi)$ is trivial since its domain is $H_1(\RDC(\gamma)) = 0$.
	Since $\VR_r(\gamma) \cong S^2$ for $ r\in (\sqrt{3},2] $ by \cite{Adamaszek2017}[Theorem 7.6], the map $H_1(\tilde\psi)$ fails to be surjective.
	Furthermore, this is another example where the kernel of $ H_0(\tilde\psi)$ contains more than the essential class of $ H_0(\RDC(\gamma)) $. \qedhere
\end{exmp}

It is not clear what curves generate nontrivial induced morphisms in degree one.
Torus knots may seem like candidates, but for equidistant samples of torus knots on flat tori, it is easy to see that the degree-one distance-to-Rips map is trivial.

\begin{exmp}
	Let $\Torus^2 = \R^2/\Z^2$ and suppose
	\[
		x_k = k \begin{pmatrix}
			p/n \\ q/n
		\end{pmatrix}, \qquad k = 1, 2, \dots, n,
	\] 
	is an equidistant sample of a $(p,q)$-torus knot.
	Then, for all $i$ and $j$,
	\[
		x_j-x_i =
		\begin{pmatrix}
			(j-i)p/n \\ (j-i)q/n
		\end{pmatrix}
		\in \Torus^2.
	\]
	Since the flat torus metric is translation invariant, the distance $d(x_i,x_j)$ depends only on $j-i \bmod n$.
	Equivalently, the distance matrix is circulant.
	The filtration of the Rips distance complex then has two different types of connected components: either isolated diagonals of vertices or diagonals of squares (potentially multiple next to each other).
	Neither type can have nontrivial $H_1$; therefore, $H_1(\RDC_r(\Gamma))=0$ for all $r$, and the degree-one distance-to-Rips maps are trivial.
\end{exmp}

\section{Application: Time Series Analysis}
\label{sec:application}

The distance complexes and associated maps are related to the cycling signature, a topological descriptor for time series segments.
The idea of the cycling signature is to decompose a recurrent invariant set into elementary recurrent motions.
The distance-to-offset map (and its variants) can be used to compute cycling signatures, but these maps also yield more detailed information.

Throughout this section, we consider homology with respect to a field.

\subsection{Background}

We introduce some background from \cite{Bauer2025cyclingSignatures}.
The idea of the cycling signature is to identify and classify certain almost-periodic segments in time series data, where almost-periodicity is quantified by a threshold distance $r \geq 0$.

Throughout this section, let $\Gamma = (x_1,\dots, x_n)$ be a time series in $\R^d$ with the Euclidean metric.
A segment of $\Gamma$ is a consecutive subsequence of $\Gamma$.
Fix a segment $\gamma=(x_i,\dots,x_j)$. 
We write
\[
	O_r^{c}(\gamma)=\{y\in\R^d\mid d(y,\gamma)\leq r\}
\]
for its closed $r$-offset.
It is \emph{$r$-cycling} if a thickening of $\gamma$ by $r$ has nontrivial first homology, i.e., if $H_1(O_r^{c}(\gamma)) \neq 0$.
Equivalently, the thickening has a nontrivial 1-cycle (and therefore an almost-closed subsegment as we explain later).

Cycling segments can be classified by their \emph{cycling signature}, which takes into account the homology classes of the thickened segment inside a so-called comparison space.

A neighborhood $Y$ of $\Gamma$ is called a \emph{comparison space}; we denote by $r_\Gamma(Y) > 0$ the maximal radius such that $O_{r}^{c}(\Gamma) \subseteq Y$ for all $r < r_\Gamma(Y)$.
Equivalently, every segment $\gamma$ of $\Gamma$ can be thickened by any $r<r_\Gamma(Y)$ without leaving $Y$. 
For every $0\leq r<r_\Gamma(Y)$, we have an inclusion map
$i_{\gamma,r}^{c}\colon O_r^{c}(\gamma)\hookrightarrow Y$.
The cycling signature of $\gamma$ at radius $r$ is
\[
    \Cyc_r(\gamma,Y)
    :=
    \im H_1(i_{\gamma,r}^{c})
    \subset H_1(Y).
\]
As the thickening radius increases, these subspaces are nested. 
Indeed, if $r\leq s<r_\Gamma(Y)$, then $ O_r^{c}(\gamma)\subseteq O_s^{c}(\gamma)\subseteq Y, $
and therefore $ \Cyc_r(\gamma,Y)\subseteq \Cyc_s(\gamma,Y). $

\begin{defn}\label{def:cycling-signature}
	The \emph{cycling signature} of a segment $\gamma$ of $\Gamma$ with respect to the comparison space $Y$ is the filtered vector space
	\[
		\Cyc(\gamma,Y)=\bigl(\Cyc_r(\gamma,Y)\bigr)_{0\leq r<r_\Gamma(Y)}.
	\]
\end{defn}

\begin{rem}
	\cref{def:cycling-signature} is slightly more general than the definition in \cite{Bauer2025cyclingSignatures} because 
	we do not require the lift to the unit tangent bundle. 
	More precisely, assume $\Gamma$ is a sampled trajectory of a vector field $f$ that never visits an equilibrium point.
	Given $x_k$, let $ v_k $ be the normalization of $ f(x_k) $. 
	Define $\rho(x_k) = (x_k,v_k)$.
	The cycling signature of $\gamma = (x_k,\dots, x_\ell)$, as defined in \cite{Bauer2025cyclingSignatures}, is the image of the map
	\[
		H_1(O(\rho(\gamma))\rightarrow \Delta Y),
	\]
	where $Y$ is a neighborhood of $\rho(\Gamma)$ and $\Delta Y$ denotes the constant filtration.
	Here, we separate the lifting step from the definition: the preceding definition applies to an arbitrary time series, and the original construction is recovered by applying it to $\rho(\Gamma)$.
\end{rem}

\subsection{Computation via the Distance Complex}

The cycling signature of a segment $\gamma$ is the image of the morphism
\[
	H_1(i_\gamma^{c})\colon H_1(O^{c}(\gamma)) \rightarrow H_1(\Delta Y),
\]
where $\Delta Y$ is the constant filtration of $Y$ and $i_\gamma^{c}$ denotes the inclusion map.

The \v{C}ech and distance filtrations developed above use strict inequalities, whereas the cycling signature just defined uses closed offsets.
For finite point clouds, it is easy to see that the distance-to-\v{C}ech map is still well-defined and induces a surjective morphism
\[
	H_0(\psi)\colon H_0(\DC^{c}(\gamma))\twoheadrightarrow H_1(\Cech^{c}(\gamma)).
\]
Using the functorial nerve theorem \cite[Theorem A]{Roll2023unified}, the cycling signature is the image of the composition
\[
	H_0(\DC^{c}(\gamma))
	\xrightarrow{H_0(\psi)}
	H_1(\Cech^{c}(\gamma))
	\xrightarrow{\simeq}
	H_1(O^{c}(\gamma))
	\xrightarrow{H_1(i_\gamma^{c})}
	H_1(\Delta Y).
\]

\subsubsection*{Algorithm and Implementation}
We first give an overview of the algorithm in \cite{Bauer2025cyclingSignatures} as implemented in \cite{Hien2024}.
There, the comparison space $Y$ is constructed as a cubical cover of $\Gamma$.
The map $H_1(\Cech^{c}(\gamma)) \rightarrow H_1(\Delta Y)$ induced by $i_\gamma^c$
is computed via a chain map of finite-dimensional chain complexes
\[
	\phi_\gamma\colon C(\Cech^{c}(\gamma))\to C(\Delta Y).
\]
The chain map is computed using the acyclic carrier that assigns each simplex of $\Cech^{c}(\gamma)$ to the cubes of the cover containing it. 
Details of this can be found in \cite{Bauer2025cyclingSignatures}, but they are not important for the following discussion.

Given an implementation of the chain map $\phi_\gamma$, cycling signatures are computed as follows.
In a preprocessing step, a basis $\alpha_1,\dots,\alpha_m$ of $H^1(Y)$ is computed.
For each segment, the first step is to produce a list of homology classes $[z_1],\dots,[z_N]$ in $H_1(\Delta Y)$, sorted by birth, that generate the cycling signature; this is where the two algorithms differ.
The cycling signature is obtained by reducing the coordinate matrix
\[
	M_{p,q}=\alpha_p(z_q)
\]
via the standard persistence reduction.
The nonzero columns form a basis of $\Cyc(\gamma,Y)$. 

For the computations below, we replace $\Cech_r^c(\gamma)$ by $\VR_{2r}^c(\gamma)$.
Note that the 1-skeleta of the two complexes are the same.
One can show that this yields the same cycling signature provided $r$ is less than the box size of the comparison space.
We describe two methods for obtaining the $z_i$:

\emph{Direct method.}
Compute generators $b_1,\dots,b_n$ of $H_1(\VR^c(\gamma))$, sorted by birth, and set
\[
	z_i=H_1(\phi_\gamma)(b_i).
\]

\emph{Distance matrix method.}
Compute generators $v_1,\dots,v_k$ of $H_0(\RDC^c(\gamma))$, sorted by birth, and set
\[
	z_i=H_1(\phi_\gamma)\bigl(H_0(\tilde\psi)(v_i)\bigr).
\]

The two methods differ in runtime and output. 
Both compute the cycling signature, but the cycles produced by the distance matrix method can be used for the additional analyses in the next section.

\subsubsection*{Runtime Considerations}
From a runtime perspective, the distance matrix method has the advantage that no reduction of a degree-one coboundary matrix is necessary.
The direct method uses standard persistence reduction for $H_1(\VR^{c}(\gamma))$, whose worst-case complexity is cubic in the number of simplices in the relevant filtration \cite{Morozov2005cubic}.
Write $\ell$ for the number of samples in the segment $\gamma$.
The distance matrix method instead computes $H_0(\RDC^{c}(\gamma))$ by connected components in a graph with $\ell(\ell+1)/2$ vertices and $O(\ell^2)$ grid edges after the distance values have been computed.
The direct method has the advantage that $H_1(\phi_\gamma)\colon H_1(\VR^{c}(\gamma)) \rightarrow H_1(\Delta Y)$ is evaluated less often, since a basis for $H_1(\VR^{c}(\gamma))$ contains at most as many elements as a generating set.

\subsubsection*{Experiments}
\Cref{fig:lorenz-dadras-runtime} shows mean runtimes for the direct method and the distance matrix method on $500$ segments for each of the segment lengths $10,20,\dots,500$ sampled from the Lorenz and Dadras systems. The code is available in the accompanying repository \cite{paperGithubRepo}.

Using our implementation, we observe that the distance matrix method is faster than the direct method for the Lorenz system. 
The situation is less clear for the Dadras system, as the direct method starts to outperform the distance matrix method for longer segments
even though, especially for longer segments, the computation of $H_0(\RDC^{c}(\gamma))$ is significantly faster than the computation of $H_1(\VR^{c}(\gamma))$.
The reason for this is that the current implementation computes
$H_1(\phi_\gamma)(H_0(\tilde\psi)(w))$ for every degree-zero persistence generator $w$ of $H_0(\RDC^{c}(\gamma))$.
This overhead can potentially be mitigated by stopping the computation once the computed image spans the homological comparison space, since the remaining elements do not influence the cycling signature.

\begin{figure}[ht]
	\centering
	\includegraphics[width=\textwidth]{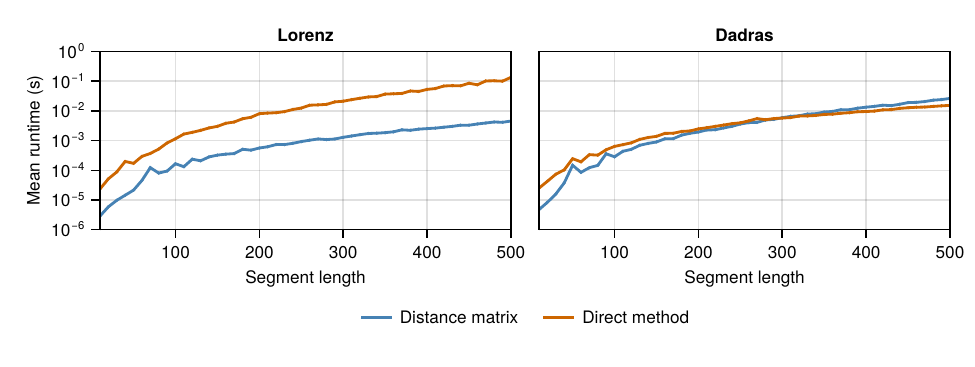}
	\caption{Mean runtime of the direct and distance matrix methods for segments sampled from the Lorenz and Dadras systems.}
	\label{fig:lorenz-dadras-runtime}
\end{figure}

\subsection{Transitions via Cycling Signatures and Temporal Cycles}

The images of vertices under the distance-to-\v{C}ech map are cycles that follow the sampled trajectory with exactly one exception: a ``jump''.
They can therefore be interpreted as almost-periodic segments.

\begin{defn}
	Let $\Gamma$ be an ordered point cloud.
	A \emph{temporal cycle} of a segment $\gamma$ of $\Gamma$ is a 1-cycle that is the image of a vertex $(i,j)$ under the distance-to-\v{C}ech map.
\end{defn}

Temporal cycles can provide a more refined description of the recurrent dynamics than the cycling signature alone. 
We demonstrate this by analyzing transitions between cycling motions in a three-wing chaotic attractor.
For computational purposes, we will consider temporal cycles with respect to the distance-to-Rips map.

\paragraph{Lü3 System}
We consider the threefold-symmetric extension of the L\"u system introduced in \cite{Anastassiou2008}.
It is given by
\begin{align}\label{eq:lu3-1}
	\dot{p}
	&=
	\frac{cp}{3}
	+\frac{(c+a)(q^2-p^2)+2pq(a-z)}{3M}
	-\frac{a}{3}(p-q)
	+\frac{qz}{3},\\
	\dot{q}
	&=
	\frac{(c-a)q-(a+z)p}{3}
	+\frac{(a-z)(p^2-q^2)+2pq(a+c)}{3M}, 	\label{eq:lu3-2} \\
	\dot{z}
	&=
	\frac{3}{2}p^2q-\frac{1}{2}q^3-bz,	\label{eq:lu3-3}
\end{align}
where $M=\sqrt{p^2+q^2}$ and $(a,b,c)=(36,3,13)$.
We numerically integrate the system over the interval $[0,10^4]$ using the adaptive Tsit5 Runge--Kutta scheme
\cite{TSITOURAS2011770}, starting from $(1,0,4.5)$, with absolute and relative tolerances $10^{-6}$. We sample the resulting trajectory at intervals of $\Delta t=0.01$. 

\begin{figure}[H]
	\centering
	\includegraphics[width=.75\linewidth]{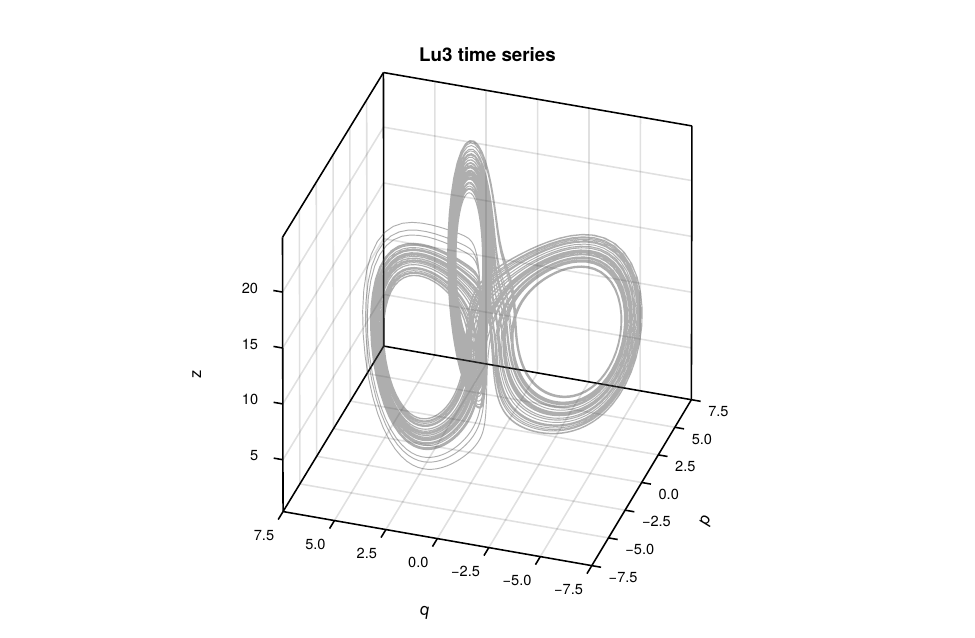}
	\caption{Numerically sampled trajectory of the Lü3 system in $(p,q,z)$-space.}
	\label{fig:lu3-time-series}
\end{figure}

\paragraph{Cycling Signatures Setup}
We follow the procedure outlined in \cite{Bauer2025cyclingSignatures}.
We lift the time series to the unit tangent bundle via
\[
	\Gamma_{utb} = \left( (x_1,v_1),\dots, (x_n,v_n) \right),
\]
where $ v_k $ is the normalized tangent vector of the vector field defined by \cref{eq:lu3-1,eq:lu3-2,eq:lu3-3} at $x_k$.
For thickenings in the unit tangent bundle, we use the dynamic distance
\[
	d_{dyn}((x,v), (y,w)) = \max\{  \norm{x-y}_2, 5\norm{v-w}_2 \}.
\]
We take as comparison space
\[
	Y = | \{ Q\mid Q = 5\,Q_1\times Q_2,\text{ where } Q_1,Q_2\in \Cub_3(\R^3) \text{ and } Q\cap \Gamma_{utb}\neq \emptyset \} |,
\]
i.e., every box is in $\R^6$, the first three components being the spatial coordinates and the last three components being the unit tangent coordinates.
We compute $ \dim H^1(Y;\F_{43}) = 3$.
To obtain cycling signatures, we sample $1000$ segments for each segment length in $\mathcal{L} = \{ 10,20,\dots, 500\}$.
This yields a total of $50\times 1000$ cycling signatures.

\paragraph{Cycling Signatures Analysis}
For simplicity, we restrict our analysis to the cycling signatures at a fixed radius $r=3.0$.
By a \emph{cycling space} we mean a subspace of $ H_1(Y) $.
We find three frequent one-dimensional cycling spaces $V_1,V_2,V_3$ and three frequent two-dimensional cycling spaces $W_1,W_2,W_3$, each containing a pair of the one-dimensional spaces; see \cref{fig:lu3-cycling-spaces}.
The three one-dimensional spaces suggest that there are three types of cycling motions. 
The three two-dimensional spaces suggest that there are transitions between the three types of cycling motions.
However, it is not clear what transitions actually occur.
For example, $W_1$ could contain transitions between $V_1$ and $V_3$ in one or both directions.

\begin{figure}[H]
	\centering
	\includegraphics[width=.95\linewidth]{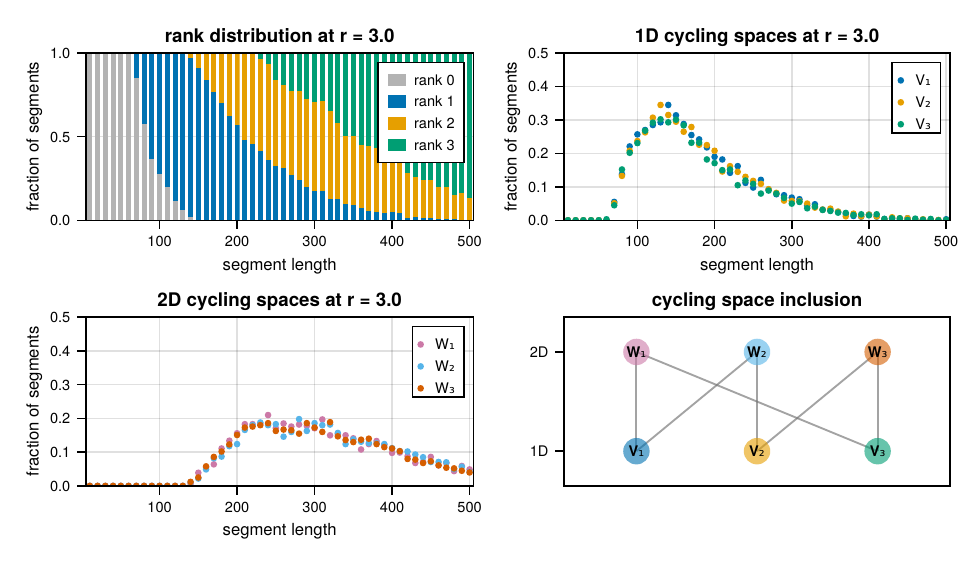}
	\caption{Cycling signatures for segments of the Lü3 trajectory at $r=3.0$. The top left panel shows the distribution of cycling ranks by segment length; the remaining panels show the three frequent one-dimensional spaces $V_i$, the three frequent two-dimensional spaces $W_i$, and their inclusion relations.}
	\label{fig:lu3-cycling-spaces}
\end{figure}

\paragraph{Transition Analysis}
For each sampled segment whose cycling space at radius $r=3.0$ is one of $W_1,W_2,W_3$, 
we compute the connected components of the distance complex to identify the temporal cycles occurring within the segment.
We then map the temporal cycles into the comparison space $Y$ and compare their images with $V_1,V_2,V_3$ to determine which type of cycling motion each temporal cycle represents.
To record transitions, we use the following notions.

\begin{defn}
	Let $v=(i,j)$ and $v'=(i',j')$ be vertices of $\DC_r(\Gamma)$, with temporal cycles $z = \psi(v)$ and $z' = \psi(v')$.
	We call $[i,j]$ the \emph{time domain} of $z$ and say that $z$ \emph{precedes} $z'$ if $i<i'$ and $j<j'$.
	In this case, $z'$ \emph{succeeds} $z$.
\end{defn}

We then classify all segments with cycling space $W_k$ according to the temporal cycles they contain.
Consider a segment with cycling space $W_k=\operatorname{span}(V_a,V_b)$.
We record a transition $V_a\to V_b$ if the segment contains a temporal cycle representing $V_a$ that precedes a temporal cycle representing $V_b$. The reverse transition is defined in the same way.
We classify the segment as bidirectional if both orderings occur and unresolved if neither ordering can be established.
We perform this classification separately for every sampled segment and count the resulting transition types for every segment length. Representative temporal cycles for segments of length $250$ are shown in \cref{fig:lu3-transition-distance-matrices}.

\begin{figure}[H]
	\centering
	\includegraphics[width=\linewidth]{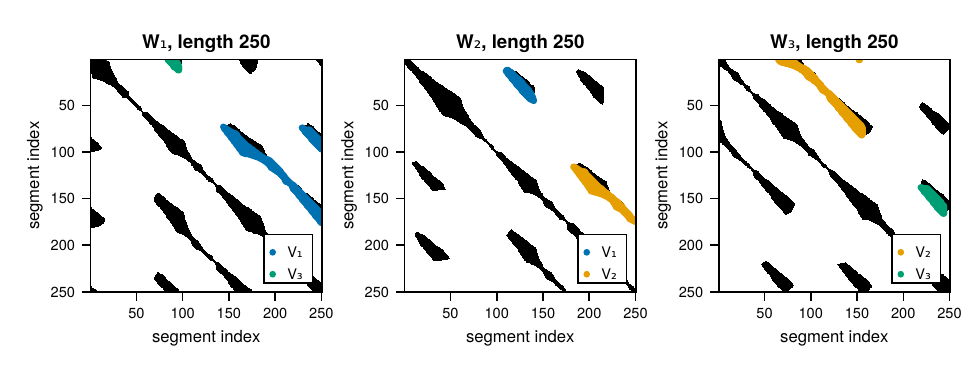}
	\caption{Temporal cycles in length-$250$ segments with cycling spaces $W_1$, $W_2$, and $W_3$.
	The grayscale background shows the distance matrix, while the colored components at $r=3.0$ represent the indicated one-dimensional cycling spaces.}
	\label{fig:lu3-transition-distance-matrices}
\end{figure}

The results show a strong directional asymmetry.
Of the $4458$ segments realizing $W_1$, $4296$ exhibit the transition
$V_3\to V_1$, while $162$ are unresolved.
Of the $4387$ segments realizing $W_2$, $4205$ exhibit $V_1\to V_2$, while
$182$ are unresolved.
Finally, of the $4324$ segments realizing $W_3$, $4131$ exhibit
$V_2\to V_3$, while $193$ are unresolved.
We observe no transitions in the opposite directions and no bidirectional
segments.
Thus, the three types of cycling motion are connected by the directed cycle
\[
	V_1\longrightarrow V_2\longrightarrow V_3\longrightarrow V_1.
\]
The transition counts by segment length and their dominant directions are shown in \cref{fig:lu3-transition-directions}.

\begin{figure}[H]
	\centering
	\includegraphics[width=\linewidth]{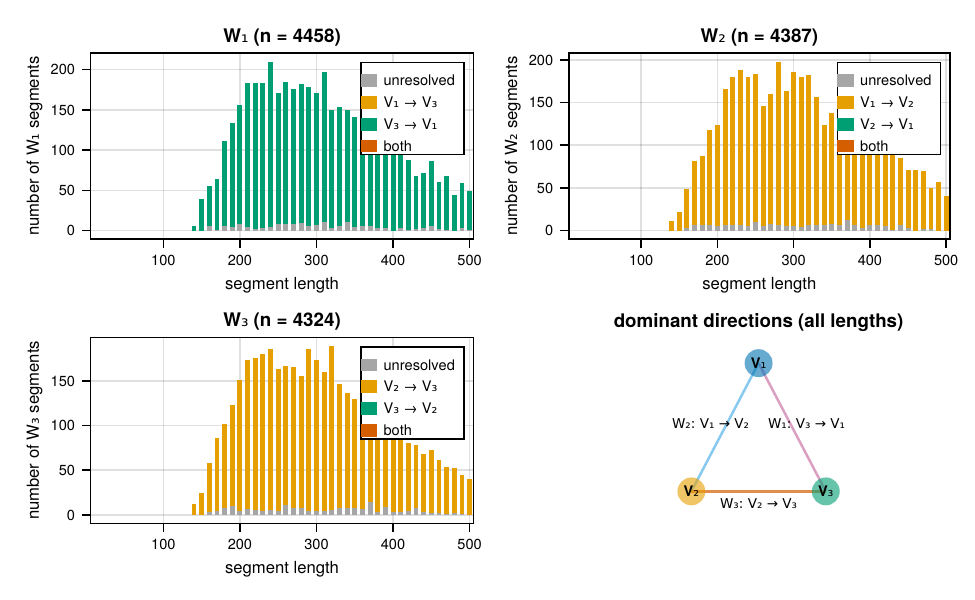}
	\caption{Transition directions for segments with cycling spaces $W_1$, $W_2$, and $W_3$. The first three panels show the transition classifications by segment length; the final panel summarizes the dominant directions over all lengths.}
	\label{fig:lu3-transition-directions}
\end{figure}

\subsection{Distance Complexes, Cycling Signatures, and Recurrence Plots}

Recurrence plots, cycling signatures, and distance complexes, together with their associated maps, provide closely related but distinct tools for analyzing dynamics from time series data.
A recurrence plot requires only a distance on the data and, once the pairwise distances are available, is independent of the ambient dimension.
It gives an inexpensive overview of recurrences at a chosen scale and has found numerous applications.
Although its visible patterns can be highly informative, their usual geometric interpretation is largely heuristic.

Computing cycling signatures additionally requires a comparison space $Y$ whose topology captures features regarded as dynamically relevant; the original algorithm itself requires no sampling condition.
The subsampling pipeline of \cite{Bauer2025cyclingSignatures} then identifies and classifies cycling motions, and the resulting signature gives a rigorous lower bound on the homological complexity of a segment relative to $Y$.
If the time series is sampled finely compared with the relevant holes in $Y$, the maps developed here recover temporal cycles and thereby enable the transition analysis above.
The main limitation is the need to construct a suitable comparison space and compute its cohomology, which can be costly in high dimensions.

Working directly with the distance complex avoids this requirement and turns the recurrence-plot filtration into an object amenable to persistent homology.
As demonstrated in \cite{Ichinomiya2023,Ichinomiya2025}, its persistence can be used for time series classification independently of the cycling-signature framework.
The distance complex requires no sampling condition and, once the distance matrix is known, is independent of the ambient dimension.
This raises the question of whether the distance-to-\v{C}ech and distance-to-Rips maps can augment such pipelines by attaching temporal and geometric information to features detected in the distance complex.

\section*{Acknowledgments}
Parts of this research were conducted while the author was a doctoral student at the Technical University of Munich and were supported by the German Research Foundation (DFG--Deutsche Forschungsgemeinschaft, TRR 109, Project ID 195170736).

\bibliographystyle{alpha}
\bibliography{refs}

\end{document}